\documentclass[11pt,letterpaper,reqno]{amsart}
\usepackage[a4paper,margin=27mm]{geometry}
\usepackage[T1]{fontenc}
\usepackage{lmodern}
\usepackage{amsmath,amssymb,amsthm,mathtools}
\usepackage[colorlinks=true,urlcolor=blue,linkcolor=blue,citecolor=blue]{hyperref}
\hypersetup{pdftitle={Sharp bounds for toral eigenfunctions and threefold additive energy on circles}}
\newtheorem{theorem}{Theorem}[section]
\newtheorem{lemma}[theorem]{Lemma}
\newtheorem{proposition}[theorem]{Proposition}
\newtheorem{corollary}[theorem]{Corollary}
\theoremstyle{remark}

\numberwithin{equation}{section}
\newcommand{\Z}{\mathbb Z}
\newcommand{\T}{\mathbb T}
\newcommand{\N}{\operatorname N}
\renewcommand{\Re}{\operatorname{Re}}
\renewcommand{\Im}{\operatorname{Im}}
\title[Threefold additive energy and toral eigenfunctions]{Sharp bounds for toral eigenfunctions and threefold additive energy on circles}
\author[G. Xie]{Guiyu Xie}
\address{Qiuzhen College, Tsinghua University, Beijing 100084, P.R. China}
\email{xiegy25@mails.tsinghua.edu.cn}
\author[C. Zhang]{Cheng Zhang}
\address{Yau Mathematical Sciences Center, Tsinghua University, Beijing 100084, P.R. China}
\email{czhang98@tsinghua.edu.cn}
\date{}
\begin{document}
\begin{abstract}
We prove the sharp bound $J_3(P_n)\le C|P_n|^3$ for the threefold
additive energy of every lattice circle $P_n=\{z\in\Z^2:|z|^2=n\}$,
with an absolute constant. This establishes the optimal bound suggested by
Krishnapur--Kurlberg--Wigman and Bombieri--Bourgain. We also prove a
uniform weak $L^6$ estimate for eigenfunctions on $\T^2$, and hence
uniform $L^p$ bounds for $1\le p<6$. This extends the classical uniform
$L^4$ bounds of Cooke and Zygmund and confirms Bourgain's conjecture in
this range.
For all $k\ge4$, we obtain the fixed power saving for the $k$-fold additive energy
$J_k(P_n)\le C|P_n|^{2k-3-\delta}$, with absolute constants $C,\delta>0$.
Applications include mean-square convergence of observables and 
the nodal length variance of arithmetic random waves.
\end{abstract}
\maketitle

\section{Introduction}
For a finite set $S\subset\mathbb R^2$ and an integer $k\ge1$, define
the $k$-fold additive energy by
\begin{equation}\label{eq:energy:def}
 J_k(S)=\#\bigl\{(z_1,\ldots,z_{2k})\in S^{2k}:
                \sum_{j=1}^kz_j=\sum_{j=k+1}^{2k}z_j\bigr\}.
\end{equation}
Tuples are ordered and repetitions are allowed. We study the lattice circles
\begin{equation}\label{eq:circle:def}
 P_n=\{z\in\Z^2:|z|^2=n\},\qquad n\ge1.
\end{equation}
Their six-point correlations determine sixth moments of the associated
exponential sums and enter the nodal length variance of arithmetic
random waves \cite{KKW}. Our first main result gives their optimal order,
uniformly over all radii. We write $|P_n|=\# P_n$. 

\begin{theorem}\label{thm:main}
There is an absolute constant $C$ such that, for every $n\ge1$,
\begin{equation}\label{eq:main}
 J_3(P_n)\le C|P_n|^3.
\end{equation}
\end{theorem}

The diagonal solutions with $(z_1,z_2,z_3)=(z_4,z_5,z_6)$ give $J_3(P_n)\ge|P_n|^3$, so the bound is sharp.
Theorem~\ref{thm:main} establishes the optimal bound suggested by
Krishnapur--Kurlberg--Wigman \cite[p.~710]{KKW} and
Bombieri--Bourgain \cite{BB}. See Section \ref{background} for related results.

Equip $\T^2=\mathbb R^2/\Z^2$ with normalized Haar measure.
For $P_n\ne\varnothing$ and every integer $k\ge3$, Fourier orthogonality
and \eqref{eq:main} give
\begin{equation}\label{eq:higher}
 J_k(P_n)=\int_{\T^2}\Bigl|\sum_{z\in P_n}e^{2\pi iz\cdot x}\Bigr|^{2k}\,dx
 \le |P_n|^{2k-6}J_3(P_n)\le C|P_n|^{2k-3}.
\end{equation}
 Theorem~\ref{thm:higher:power}
improves this consequence by a fixed power:
$J_k(P_n)\le C|P_n|^{2k-3-\delta}$ for every integer $k\ge4$,
with absolute constants $C,\delta>0$.

Every eigenfunction of $-\Delta$ on $\T^2$ with eigenvalue $4\pi^2n$,
$n\ge1$, has the form
\begin{equation}\label{eq:app:weighted-sum}
 F(x)=\sum_{z\in P_n}a_z e^{2\pi iz\cdot x},\qquad a_z\in\mathbb C.
\end{equation}
Parseval's identity gives
\begin{equation}\label{eq:app:parseval}
 \|F\|_2^2=\sum_{z\in P_n}|a_z|^2.
\end{equation}
The energy estimates correspond to equal coefficients. Our second main
result allows arbitrary coefficients and controls the weak endpoint.
For a measurable function $h$ on $\T^2$, let
\begin{equation}\label{eq:lp:weak-norm}
 \|h\|_{q,\infty}:=\sup_{t>0}t\,|\{x\in\T^2:|h(x)|>t\}|^{1/q}
 \qquad(q>1).
\end{equation}

\begin{theorem}\label{thm:lp:weak}
There is an absolute constant $C_{\mathrm w}\ge1$ such that every
eigenfunction on $\mathbb{T}^2$ satisfies
\begin{equation}\label{eq:lp:weak}
 \|F\|_{6,\infty}\le C_{\mathrm w}\|F\|_2.
\end{equation}
For every $1\le p<6$, there is a constant $C_p$, depending only on $p$,
such that
\begin{equation}\label{eq:lp:subcritical}
 \|F\|_p\le C_p\|F\|_2.
\end{equation}
\end{theorem}

This extends the classical uniform $L^4$ bounds of Cooke \cite{Cooke}
and Zygmund \cite{Zygmund}. Bourgain \cite{BourgainMoment} conjectured
\begin{equation}\label{T2uniform}
 \|F\|_p\le C_p\|F\|_2,\qquad 1\le p<\infty.
\end{equation}
Theorem~\ref{thm:lp:weak} confirms \eqref{T2uniform} for $1\le p<6$
and reaches $p=6$ in weak form.

\subsection{Background and related work}\label{background}

Krishnapur, Kurlberg and Wigman \cite[Theorem~2.2]{KKW} recorded
Bourgain's bound $J_3(P_n)=o(|P_n|^4)$ and conjectured
$J_3(P_n)\le C_\varepsilon|P_n|^{3+\varepsilon}$ for every
$\varepsilon>0$. Bombieri and Bourgain \cite{BB} proved
$J_3(P_n)\le C|P_n|^{7/2}$ as a consequence of the Szemer\'edi-Trotter Theorem. Both papers cautiously speculated on the possibility that the stronger $O(|P_n|^{3})$ bound might hold. For the connection with the Unit Distance Conjecture, see
Bourgain--Demeter \cite[Questions~2.13--2.14 and p.~363]{BD15} and
\cite{OpenAIUnitDistance}. Related estimates appear in Li
\cite{LiEffective}, Guth--Maldague--Wang \cite{GMW}, Mudgal
\cite{MudgalSpheres,MudgalVMVT}, and Blomer--Br\"udern
\cite{BlomerBrudern}.

Let
\[
 \mathcal P=\{(t,t^2):t\in\mathbb R\},\qquad
 \mathbb S^1=\{(x,y)\in\mathbb R^2:x^2+y^2=1\}.
\]
Cushman, Demeter and Wu \cite[Theorem~1.1]{CDW} proved
$J_3(S)\le C_\varepsilon|S|^{3+\varepsilon}$ for every finite set on the
graph of a strictly convex real function, with $C_\varepsilon$ independent
of the function. For finite $S\subset\mathcal P\cap\mathbb Q^2$ or
$S\subset\mathbb S^1\cap\mathbb Q^2$, the stronger estimate
\begin{equation}\label{eq:intro:LY}
 J_3(S)\le C|S|^3\log(2+|S|)
\end{equation}
holds with an absolute constant. On the parabola, this follows from
Deng--Fan--Guo--Guo--Luo \cite[Theorem~1.2]{DFGGL}. Li--Yung \cite[Theorem~1.1]{LY} independently established
\eqref{eq:intro:LY} for both curves. These logarithmic estimates concern
rational points, whereas the theorem of \cite{CDW} allows arbitrary
real points.

The logarithm in \eqref{eq:intro:LY} is necessary on the parabola.
For $S_N=\{(m,m^2):m=1,\ldots,N\}$, Rogovskaya
\cite[p.~78--84]{Rogovskaya} proved the quadratic Vinogradov mean value
asymptotic
\[
 J_3(S_N)=\frac{18}{\pi^2}N^3\log N+O(N^3).
\]
See also Blomer--Br\"udern \cite{BlomerBrudern}. Thus the logarithmic
loss is unavoidable for general parabola point sets, but
Theorem~\ref{thm:main} removes it for lattice circles.

The weak $L^6$ estimate of \cite[Theorem~1.1]{DFGGL} also implies that,
for every finite $S\subset\mathcal P\cap\mathbb Q^2$,
\begin{equation}\label{eq:intro:parabola-higher}
 J_k(S)\le C|S|^{2k-3}\qquad(k\ge4),
\end{equation}
with an absolute constant $C$. The exponent $2k-3$ is optimal for each
fixed $k\ge4$, since
\[
 J_k(S_N)\ge c_kN^{2k-3},
\]
where $c_k>0$ depends only on $k$.
See Wooley \cite[(1.5), p.~1576]{WooleyEC} and Mudgal
\cite[p.~81]{MudgalVMVT} for this lower bound.

For  lattice circles, Bourgain's conjecture
\eqref{T2uniform} predicts
\begin{equation}\label{eq:higher:conjecture}
 J_k(P_n)\le C_k|P_n|^k\qquad(k\ge2).
\end{equation}
For $k=2$, the classical four-point pairing argument gives
$J_2(P_n)=3|P_n|^2-3|P_n|$ \cite[(72)]{KKW}. For $k=3$,
Theorem~\ref{thm:main} proves \eqref{eq:higher:conjecture}. For $k\ge4$,
the conjecture remains open. Theorem~\ref{thm:higher:power} gives a
fixed power saving over \eqref{eq:higher}, in contrast to the sharp
parabola bound \eqref{eq:intro:parabola-higher}. 
Proving
\eqref{eq:higher:conjecture} may require new ideas.

Bombieri--Bourgain \cite[Theorem~17 and Remark~15]{BB} and Bourgain
\cite[Lemma~4]{BourgainRandom} give a set
$\mathcal S_0\subset\mathcal S:=\{n\ge1:P_n\ne\emptyset\}$ of relative asymptotic
density one,
\begin{equation}\label{eq:higher:density}
 \frac{\#(\mathcal S_0\cap[1,X])}{\#(\mathcal S\cap[1,X])}
 \longrightarrow1\qquad(X\longrightarrow\infty),
\end{equation}
such that $|P_n|\to\infty$ along $\mathcal S_0$ and, for every fixed $k\ge2$,
\begin{equation}\label{eq:higher:generic}
 J_k(P_n)=(2k-1)!!\,|P_n|^k+O_k(|P_n|^{k-\eta_k}),\qquad \eta_k>0,
 \quad n\longrightarrow\infty,\ n\in\mathcal S_0.
\end{equation}
See also Wigman--Yesha \cite[\S2.2, (2.6)]{WY}.
Here $(2k-1)!!=(2k)!/(2^k k!)$.

For eigenfunctions in \eqref{eq:app:weighted-sum}, Cooke \cite{Cooke}
and Zygmund \cite{Zygmund} proved $\|F\|_p\le C\|F\|_2$ for $1\le p\le4$.
Cauchy--Schwarz and the classical divisor bound give
\[
 \|F\|_\infty\le |P_n|^{1/2}\|F\|_2
 \le C_\varepsilon n^\varepsilon\|F\|_2\qquad(\varepsilon>0),
\]
see \cite{ETT} for the representation-function estimates.

Guth--Maldague--Wang \cite[Theorem~1.1 and Appendix]{GMW} established
\[
 \|F\|_6\le C[\log(2+n)]^A\|F\|_2
\]
with absolute constants $A,C>0$.
Li--Yung \cite[Theorem~1.2]{LY} improve this to
\begin{equation}\label{eq:intro:LY-weighted}
 \|F\|_6\le C[\log(2+|P_n|)]^{1/6}\|F\|_2
 \le C\Bigl[\frac{\log(3+n)}{\log\log(3+n)}\Bigr]^{1/6}\|F\|_2.
\end{equation}
Interpolation with the $L^4$ and $L^\infty$ bounds gives
\[
 \|F\|_p\le C_p\|F\|_2
 \begin{cases}
 [\log(2+|P_n|)]^{1/2-2/p},&4<p\le6,\\[1mm]
 |P_n|^{1/2-3/p}[\log(2+|P_n|)]^{1/p},&6<p<\infty.
 \end{cases}
\]
Theorem~\ref{thm:lp:weak} removes the loss for $4<p<6$ and gives a
uniform weak $L^6$ bound.

Related estimates for functions with Fourier support in thin annuli on
$\T^2$ are due to Bourgain--Burq--Zworski \cite[Proposition~2.4]{BBZ},
Germain--Rydin Myerson \cite[Theorems~1.1--1.2]{GermainMyerson},
Blair--Huang--Sogge \cite[Theorem~6.1]{BHS},
Pezzi \cite[Theorem~1.3]{PezziProjection},
Germain--Rydin Myerson--Pezzi \cite[Theorem~1.4]{GRMP}, and
Demeter--Germain \cite[Theorem~1.3]{DemeterGermain}.
For eigenfunction bounds on higher-dimensional tori, Bourgain
\cite{B1993} formulated the Discrete Restriction Conjecture. See also
Bourgain--Demeter \cite[Conjecture~2.6 and Theorem~2.7]{BD15}.

\subsection{Arithmetic of lattice circles}\label{subsec:circle-parabola}

Both Theorems~\ref{thm:main} and~\ref{thm:lp:weak} use arithmetic
properties of lattice circles. The proof of Theorem~\ref{thm:lp:weak}
works for arbitrary coefficients supported on any subset of the circle,
whereas that of Theorem~\ref{thm:main} additionally relies on exact point
counts and the inclusion of all admissible rotation images in the
complete circle.

The proof of Theorem~\ref{thm:lp:weak} uses the graph estimate in
Lemma~\ref{lem:lp:graph-weak}, which gives a uniform weak $L^6$ bound
for exponential sums with arbitrary coefficients supported on integer
points of a $3$-adic graph satisfying the four-point identity
\eqref{eq:lp:second-difference}. To reduce \eqref{eq:lp:weak} to this
estimate, divisibility of sums of two squares allows common powers of
$3$ to be removed from the coordinates. By multiplicativity of the norm,
the resulting frequency set is then mapped into a circle of integer
radius coprime to $3$, as in \eqref{eq:lp:integer-circle}. Decomposition
into residue classes modulo $3$, followed by lattice-preserving
rotations and suitable translations and rescalings, gives the implicit
equation \eqref{eq:lp:circle-conic}. Completeness of the $3$-adic integers
and the contraction mapping principle uniquely determine the graph
\eqref{eq:lp:circle-model}. This graph satisfies the exact four-point
identity \eqref{eq:lp:second-difference}, as verified in
\eqref{eq:lp:circle-increment}--\eqref{eq:lp:circle-difference}.
The invariance in \eqref{eq:lp:affine-distribution} transfers the estimate
of Lemma~\ref{lem:lp:graph-weak} to each residue class. Combining these
estimates proves \eqref{eq:lp:weak}. The remaining $L^p$ bounds follow by
integrating the distribution function as in
\eqref{eq:lp:subcritical-integration}.

The proof of Theorem~\ref{thm:main} uses Theorem~\ref{thm:lp:weak}
through the logarithmic subset estimate \eqref{eq:log-energy}, but
requires further arithmetic structure. After the reduction
\eqref{eq:reduction:circle}, unique factorization in the Gaussian
integers gives the exact point count \eqref{eq:product} and the
prime-power decomposition \eqref{eq:split:sets}. Central symmetry and
divisibility relations in the Gaussian integers then yield the recurrence \eqref{eq:recurrence}
in Proposition~\ref{prop:recurrence}. The common norm gives the support
condition \eqref{eq:local:directions} through
\eqref{eq:local:circle:identity}. Together with
Lemma~\ref{lem:local:cubic} and \eqref{eq:log-energy}, this yields
Proposition~\ref{prop:local}, reducing the recurrence error to control
of a residue-class proportion times a logarithmic factor.

Completeness of the lattice circle is essential in
Proposition~\ref{prop:residues}: the admissible rotations in
\eqref{eq:packing:images} preserve integrality and norm, so all images
remain in the circle. Injectivity, established by
\eqref{eq:packing:determinant}, then gives the residue-class bound
\eqref{eq:rho}. Bourgain's six-term estimate
\cite[Lemma~5]{BourgainRandom}, recalled in
Lemma~\ref{lem:unit:nondegenerate}, bounds the number of nondegenerate
zero-sum relations. Parity
\eqref{eq:unit:parity} and the four-point pairing relation implied by
\eqref{eq:unit:four-points} handle the degenerate ones, giving
Proposition~\ref{prop:unit}. Finally, choose a divisor
maximizing the ratio of the threefold additive energy to the cube of the
point count, as in \eqref{eq:closure:maximum}. If this ratio exceeds a
fixed constant, Proposition~\ref{prop:unit} controls the logarithmic
factor in terms of the number of distinct prime factors, while
Proposition~\ref{prop:residues} provides the residue-class bound
\eqref{eq:rho}. Their product is therefore uniformly bounded by
\eqref{eq:closure:residue}--\eqref{eq:closure:range}, removing the
logarithmic loss and closing the recurrence.

For higher energy estimates in Theorem~\ref{thm:higher:power},
Proposition~\ref{prop:higher:rank-count} extends the rotation count to
several prime divisibility conditions, yielding
Lemma~\ref{lem:higher:weighted}. Factorization and conjugation give the
eight-point estimate in Proposition~\ref{prop:higher:eight-point}.
Combining these estimates with Theorem~\ref{thm:main} gives
Theorem~\ref{thm:higher:power}.

\subsection{Organization}\label{subsec:organization}
Section~\ref{sec:prelim} records Fourier projection estimates and
Gaussian-integer preliminaries. Section~\ref{sec:eigenfunction-lp}
proves Theorem~\ref{thm:lp:weak} and the logarithmic sixth-moment bound.
Section~\ref{sec:closure} proves Theorem~\ref{thm:main} from the
prime-power recurrence, local estimate, residue-class count, and
six-term relation bound. Section~\ref{sec:higher-energy} proves
Theorem~\ref{thm:higher:power} using weighted counting estimates.
Section~\ref{sec:applications} gives applications to endpoint
eigenfunction estimates, triple sumsets, mean-square convergence of
observables, and nodal length fluctuations.

\subsection*{Acknowledgments}
The authors would like to thank Shukun Wu for helpful discussions. The authors are partially supported by National Key R\&D Program of China
2024YFA1015300 and NSFC Grant 12371097

\subsection*{Statements and Declarations}
Data sharing not applicable to this article as no datasets were generated or analyzed during the current study. The authors have no relevant financial or non-financial interests to disclose. The mathematical content of this paper and all of the writing were generated solely by the authors. ChatGPT was used during the editorial process to identify mathematical and typographical errors.

\section{Preliminaries}\label{sec:prelim}
We collect Fourier projection estimates and describe the Gaussian-integer
factorization of lattice circles.

\subsection{Fourier projections and positivity}\label{subsec:fourier-tools}
Let $M_n=|P_n|$ for $n\ge1$, with $P_n$ as in \eqref{eq:circle:def}.
All logarithms are natural. Unless fixed explicitly, $C$ denotes an
absolute constant that may change from line to line.
We identify $\Z^2$ with $\Z[i]$ and write $\N(z)=z\bar z$.
For a set $E$, $\mathbf1_E$ denotes its indicator.
For an integer $q\ge1$, $\sum_{r\bmod q}$ denotes summation over
all residue classes in $\Z/q\Z$.
For a finite set $S\subset\Z^2$, let
\begin{equation}\label{eq:fourier:sum}
 f_S(x)=\sum_{z\in S}e^{2\pi i z\cdot x}.
\end{equation}
For $h\in L^1(\T^2)$, let
\begin{equation}\label{eq:fourier:coefficients}
 \widehat h(v)=\int_{\T^2}h(x)e^{-2\pi i v\cdot x}\,dx,
 \qquad v\in\Z^2.
\end{equation}
Unless indicated otherwise, integrals and $L^q$ norms are on $\T^2$.
The weak $L^q$ space $L^{q,\infty}(\T^2)$ consists of functions with
finite quasi-norm \eqref{eq:lp:weak-norm}.
For finite $S\subset\Z[i]$ and integers $k\ge1$, Fourier orthogonality,
\eqref{eq:energy:def}, and \eqref{eq:fourier:sum} give
\begin{equation}\label{eq:energy}
 J_k(S)=\int_{\T^2}|f_S|^{2k}.
\end{equation}
The same expansion with $k=1$ gives
\begin{equation}\label{eq:fourier:counts}
 \widehat{|f_S|^2}(v)=\#\{(u,w)\in S^2:u-w=v\}\ge0.
\end{equation}
If $S=-S$, then $\overline{f_S}=f_{-S}=f_S$, so $f_S$ is real and
$f_S^2=|f_S|^2$ in \eqref{eq:fourier:counts}.
For the sets in \eqref{eq:circle:def}, multiplication by $i$
has orbits of size four, and conjugation preserves the norm. Hence
\begin{equation}\label{eq:circle:symmetries}
 iP_n=P_n=-P_n=\overline{P_n},\qquad
 M_n\ge4\quad(P_n\ne\varnothing).
\end{equation}

For a full-rank sublattice $\Lambda\subset\Z^2$, let
$\Lambda^\perp=\{y\in\T^2:e^{2\pi i v\cdot y}=1\text{ for all }v\in\Lambda\}$.
For a trigonometric polynomial $h$, character orthogonality on this
finite group gives
\begin{equation}\label{eq:projection:identity}
 E_\Lambda h(x):=\sum_{v\in\Lambda}\widehat h(v)e^{2\pi i v\cdot x}
 =\frac1{|\Lambda^\perp|}\sum_{y\in\Lambda^\perp}h(x+y).
\end{equation}
The averaging formula defines $E_\Lambda$ on every $L^q$,
$1\le q\le\infty$. It is self-adjoint on $L^2$, preserves integrals,
and, for every full-rank sublattice $\Lambda'\subset\Lambda$, satisfies
\begin{equation}\label{eq:projection:properties}
 \begin{gathered}
 E_\Lambda h\ge0\quad(h\ge0),\qquad
 \|E_\Lambda h\|_q\le\|h\|_q\quad(1\le q\le\infty),\\
 E_{\Lambda'}E_\Lambda=E_{\Lambda'}\quad(\Lambda'\subset\Lambda).
 \end{gathered}
\end{equation}
For bounded $g$ satisfying $g(x+y)=g(x)$ for $y\in\Lambda^\perp$
and integrable $h$, averaging in \eqref{eq:projection:identity} gives
\begin{equation}\label{eq:projection:module}
 E_\Lambda(gh)=gE_\Lambda h,\qquad
 \int gh=\int gE_\Lambda h.
\end{equation}
For every nonsingular integer $2\times2$ matrix $A$, the map $A^T:\T^2\to\T^2$
preserves Haar measure. Thus
\begin{equation}\label{eq:scaling}
 f_{AS}(x)=f_S(A^Tx),\qquad
 \|h(A^T\cdot)\|_q=\|h\|_q\quad(1\le q\le\infty).
\end{equation}
For $0\ne\gamma=a+ib\in\Z[i]$, multiplication by $\gamma$ has matrix
\[
 A_\gamma=\begin{pmatrix}a&-b\\ b&a\end{pmatrix},\qquad
 \det A_\gamma=a^2+b^2=\N(\gamma)>0.
\]

\begin{samepage}
The strong and weak estimates below follow from complex and real interpolation,
respectively \cite[Chapter~5]{BL}. See also
\cite[Lemma~3.6]{DFGGL} for the weak estimate.
\begin{lemma}\label{lem:projection:summation}
Let $(\Pi_\nu)_{\nu\in I}$ be a finite family of Fourier projections on
$\T^2$ with pairwise disjoint frequency supports and
$\|\Pi_\nu\|_{L^\infty\to L^\infty}\le2$.
For $f_\nu=\Pi_\nu f_\nu$, we have, with $C$ absolute,
\begin{align}
 \Bigl\|\sum_\nu f_\nu\Bigr\|_3^3
 &\le2\Bigl(\sum_\nu\|f_\nu\|_3^{3/2}\Bigr)^2,
 \label{eq:one:scale}\\
 \Bigl\|\sum_\nu f_\nu\Bigr\|_{3,\infty}
 &\le C\Bigl(\sum_\nu\|f_\nu\|_{3,\infty}^{3/2}\Bigr)^{2/3},
 \label{eq:projection:sum:weak}
\end{align}
\end{lemma}
\end{samepage}

\begin{samepage}
By \eqref{eq:projection:identity}, the operators $E_{\Lambda_j}$ below
are conditional expectations for a decreasing filtration.
The following martingale estimate follows from Burkholder's square-function inequality
\cite[Theorem~3.1]{Burkholder}, by duality, real interpolation
\cite[Chapter~5]{BL} and the triangle inequality for an
equivalent norm on $L^{3/2,\infty}$. See also \cite[Lemma~3.1]{DFGGL}.
\begin{lemma}\label{lem:projection:martingale}
Let $\Z^2=\Lambda_0\supset\Lambda_1\supset\cdots\supset\Lambda_J$
be full-rank sublattices, with $J\ge1$. For every
$h\in L^{3,\infty}(\T^2)$,
\begin{equation}\label{eq:projection:martingale}
 \|h\|_{3,\infty}^2
 \le C\|E_{\Lambda_J}h\|_{3,\infty}^2
   +C\sum_{j=0}^{J-1}
       \|(E_{\Lambda_j}-E_{\Lambda_{j+1}})h\|_{3,\infty}^2,
\end{equation}
where $C$ is absolute and independent of $J$ and the sublattices.
\end{lemma}
\end{samepage}

\subsection{Gaussian integers and reduction to split primes}\label{subsec:gaussian}
For the standard facts used below, see Hardy and Wright
\cite[Chapter~XII and \S15.1]{HardyWright}.
A \emph{rational prime} is a positive prime in $\Z$.
The units of $\Z[i]$ are $\{1,-1,i,-i\}$. Two nonzero elements are
\emph{associates} if they differ by a unit factor, and \emph{nonassociate}
otherwise. A \emph{Gaussian prime} is a nonzero nonunit $\varpi$ such that,
for $u,v\in\Z[i]$, $\varpi\mid uv$ implies $\varpi\mid u$ or
$\varpi\mid v$.
The ring $\Z[i]$ has unique factorization up to units and order.
In particular, nonassociate Gaussian primes are coprime.
The norm satisfies $\N(a+ib)=a^2+b^2$ and $\N(uv)=\N(u)\N(v)$.
Rational primes factor as follows:
\begin{equation}\label{eq:gaussian:primes}
 \begin{aligned}
  2&=-i(1+i)^2,\\
  p\equiv1\pmod4&:\quad p=\varpi\bar\varpi,\quad
       \N(\varpi)=p,\quad \varpi,\bar\varpi\text{ nonassociate Gaussian primes},\\
  q\equiv3\pmod4&:\quad q\text{ is a Gaussian prime}.
 \end{aligned}
\end{equation}
We call rational primes $p\equiv1\pmod4$ \emph{split primes}.
The factor $\varpi$ may be chosen with positive real and imaginary parts.
Divisibility by an ordinary integer is coordinatewise:
\begin{equation}\label{eq:gaussian:integer-divisibility}
 k\mid(a+ib)\text{ in }\Z[i]
 \quad\Longleftrightarrow\quad k\mid a\text{ and }k\mid b\text{ in }\Z
 \qquad(k\in\Z\setminus\{0\}).
\end{equation}
If $\N(\varpi)=p$ and $\varpi\mid z$, then $p\mid\N(z)$.
For a split prime $p=\varpi\bar\varpi$, if $\N(z)=m$ and $p\nmid m$,
then $m^{-1}\bar z$ is an inverse of $z$ modulo $p$, $\varpi$, and
$\bar\varpi$.

Factor a positive integer as
\begin{equation}\label{eq:reduction:factor}
 n=2^a\prod_{p\equiv1\pmod4}p^{\nu_p}
          \prod_{q\equiv3\pmod4}q^{f_q}.
\end{equation}
Conjugation fixes each $q\equiv3\pmod4$ and preserves the exponent of
$1+i$. Thus unique factorization in $z\bar z=n$ shows that $P_n$ is
empty if some $f_q$ is odd, and otherwise gives
\begin{equation}\label{eq:reduction:circle}
 P_n=\gamma P_{n_0},\qquad
 \gamma=(1+i)^a\prod_{q\equiv3\pmod4}q^{f_q/2},\qquad
 n_0=\prod_{p\equiv1\pmod4}p^{\nu_p}.
\end{equation}
Multiplication by $\gamma\ne0$ is a bijection preserving additive
relations, so $M_n=M_{n_0}$ and $J_3(P_n)=J_3(P_{n_0})$.
It therefore suffices to treat products of split prime powers.

Write $n=\prod_{j=1}^r p_j^{\nu_j}$ with distinct split primes $p_j$
and choose $\varpi_j$ with $\N(\varpi_j)=p_j$.
In $z\bar z=n$, the exponents of $\varpi_j$ and $\bar\varpi_j$ in $z$
sum to $\nu_j$. Hence
\begin{equation}\label{eq:product}
 \begin{gathered}
 P_n=\bigl\{\epsilon\prod_{j=1}^r
        \varpi_j^{k_j}\bar\varpi_j^{\nu_j-k_j}:
        \epsilon\in\{1,-1,i,-i\},\ 0\le k_j\le\nu_j\bigr\},\\
 M_n=4\prod_{j=1}^r(\nu_j+1).
 \end{gathered}
\end{equation}
The representation is unique, so the four units and the $\nu_j+1$
choices of each exponent give the point count. Empty products give
$P_1=\{1,-1,i,-i\}$ and $M_1=4$.

\section{Proof of Theorem~\ref{thm:lp:weak}}
\label{sec:eigenfunction-lp}
We first deduce the theorem and a logarithmic sixth-moment bound from a
$3$-adic graph estimate. We then prove the graph estimate using a bilinear
lemma, whose proof is given last by a truncation argument.

For finite $S\subset\Z^2$ and coefficients $a=(a_\xi)_{\xi\in S}$, let
\begin{equation}\label{eq:lp:integer-sum}
 U_{S,a}(x)=\sum_{\xi\in S}a_\xi e^{2\pi i\xi\cdot x},
 \qquad \|a\|_2^2=\sum_{\xi\in S}|a_\xi|^2.
\end{equation}
Let $\zeta\in\Z^2$ and let $M$ be a nonsingular integer $2\times2$ matrix.
Assign coefficient $a_\xi$ to $\zeta+M\xi$. Then
\begin{equation}\label{eq:lp:affine-invariance}
 U_{\zeta+MS,a}(x)=\sum_{\xi\in S}a_\xi e^{2\pi i(\zeta+M\xi)\cdot x}
 =e^{2\pi i\zeta\cdot x}U_{S,a}(M^Tx).
\end{equation}
Then \eqref{eq:scaling} and \eqref{eq:lp:affine-invariance} give
\begin{equation}\label{eq:lp:affine-distribution}
 \begin{gathered}
 |\{|U_{\zeta+MS,a}|>\alpha\}|
 =|\{|U_{S,a}|>\alpha\}|\qquad(\alpha>0),\\
 \|U_{\zeta+MS,a}\|_{p}=\|U_{S,a}\|_{p}\quad(1\le p\le\infty),
 \qquad
 \|U_{\zeta+MS,a}\|_{p,\infty}=\|U_{S,a}\|_{p,\infty}\quad(p>1).
 \end{gathered}
\end{equation}
For $0\ne n\in\Z$, let $v_3(n)=\max\{k\ge0:3^k\mid n\}$, and extend this
to $\mathbb Q^\times$ by $v_3(a/b)=v_3(a)-v_3(b)$. Set
$v_3(0)=+\infty$ and $|t|_3=3^{-v_3(t)}$, with $|0|_3=0$.
The field $\mathbb Q_3$ is the completion of $\mathbb Q$ for the distance
$|s-t|_3$. The valuation and absolute value extend to $\mathbb Q_3$.
Its ring of integers is the complete ring
\[
 \mathbb Z_3=\{t\in\mathbb Q_3:|t|_3\le1\},
 \qquad \Z\subset\mathbb Z_3.
\]
We use the properties
\begin{equation}\label{eq:lp:adic-properties}
 \begin{gathered}
 v_3(ab)=v_3(a)+v_3(b),\qquad
 v_3(a+b)\ge\min\{v_3(a),v_3(b)\},\\
 |ab|_3=|a|_3|b|_3,\qquad
 |a+b|_3\le\max\{|a|_3,|b|_3\},\\
 a\equiv b\pmod{3^k}
 \ \Longleftrightarrow\ a-b\in3^k\mathbb Z_3
 \ \Longleftrightarrow\ |a-b|_3\le3^{-k}
 \qquad(a,b\in\mathbb Z_3,\ k\ge0).
 \end{gathered}
\end{equation}
Here
$v_3(0)=+\infty$. For ordinary integers the
congruence in \eqref{eq:lp:adic-properties} is the usual one, since
$\Z\cap3^k\mathbb Z_3=3^k\Z$.
An element $a\in\mathbb Z_3$ is a \emph{unit} if it is invertible in
$\mathbb Z_3$, equivalently $v_3(a)=0$, $|a|_3=1$, or
$a\not\equiv0\pmod3$. In this case $|a^{-1}|_3=1$.
Every nonzero $h\in\mathbb Z_3$ has the form $h=3^ku$, where
$k=v_3(h)\ge0$ and $u$ is a unit.

For a map $\varphi:\mathbb Z_3\to\mathbb Z_3$, write
\[
 \Gamma_\varphi=\{(t,\varphi(t)):t\in\mathbb Z_3\}\subset\mathbb Z_3^2.
\]
A set $S\subset\Z^2$ lies on this graph precisely when
$S\subset\Gamma_\varphi\cap\Z^2$, that is,
$s=\varphi(t)$ in $\mathbb Z_3$ for every $(t,s)\in S$.
We shall consider maps satisfying
\begin{equation}\label{eq:lp:second-difference}
 |\varphi(u)+\varphi(v)-\varphi(w)-\varphi(z)|_3
 =|u-w|_3|u-z|_3
 \quad(u,v,w,z\in\mathbb Z_3,\ u+v=w+z).
\end{equation}
The following estimate will be proved in
Section~\ref{subsec:lp:graph-proof}.
\begin{lemma}[Graph estimate]\label{lem:lp:graph-weak}
For every finite $S\subset\Gamma_\varphi\cap\Z^2$, where $\varphi$ satisfies
\eqref{eq:lp:second-difference}, and every coefficient vector $a$ in
\eqref{eq:lp:integer-sum},
\begin{equation}\label{eq:lp:graph-bound}
 \|U_{S,a}\|_{6,\infty}\le C\|a\|_2,
\end{equation}
with an absolute constant.
\end{lemma}

\begin{proof}[Proof of Theorem~\ref{thm:lp:weak}]
We reduce each lattice circle to four integer subsets of a graph satisfying
\eqref{eq:lp:second-difference}, and then apply Lemma~\ref{lem:lp:graph-weak}.
The zero function and constants are immediate. If $P_n\ne\varnothing$,
then $3\mid x^2+y^2$ implies $3\mid x,y$, because squares modulo $3$ are
$0$ and $1$. Repeatedly dividing both coordinates by $3$ gives
\begin{equation}\label{eq:lp:integer-circle}
 \begin{gathered}
 n=3^{2k}m,\qquad 3\nmid m,\qquad P_n=3^kP_m,\\
 z_0\in P_m,\qquad
 S=\overline{z_0}P_m\subset\{(x,y)\in\Z^2:x^2+y^2=m^2\}.
 \end{gathered}
\end{equation}
For $z_0=a_0+ib_0$, multiplication by $\overline{z_0}$ has integer matrix
$\bigl(\begin{smallmatrix}a_0&b_0\\-b_0&a_0\end{smallmatrix}\bigr)$ and
 determinant $a_0^2+b_0^2=m>0$.
Thus \eqref{eq:lp:affine-distribution} reduces the proof to $U_{S,a}$, with
the coefficients reindexed.
Since $m^2\equiv1\pmod3$, the four possible classes of $(x,y)$ are
$(\pm m,0),(0,\pm m)$. An integer rotation sends each to $(0,m)$ and
preserves the circle. In that class, $3\mid x$ and $3\nmid y+m$.
The identity $(y-m)(y+m)=-x^2$ then gives $9\mid y-m$. Consequently
\begin{equation}\label{eq:lp:circle-conic}
 \begin{gathered}
 x=3t,\qquad y=m+9s,\qquad t,s\in\Z,\\
 (3t)^2+(m+9s)^2=m^2
 \quad\Longleftrightarrow\quad t^2+2ms+9s^2=0.
 \end{gathered}
\end{equation}
For fixed $t\in\mathbb Z_3$, define
$T_t(s)=-(t^2+9s^2)/(2m)$. Since $2m$ is a unit,
\[
 |T_t(s)|_3\le\max\{|t|_3^2,3^{-2}|s|_3^2\}\le1,
 \qquad
 |T_t(s)-T_t(s')|_3
 =3^{-2}|s-s'|_3|s+s'|_3\le3^{-2}|s-s'|_3.
\]
Thus $T_t$ maps $\mathbb Z_3$ into itself and is a contraction. Iterating
$s_0=0$, $s_{\ell+1}=T_t(s_\ell)$ gives
$|s_{\ell+1}-s_\ell|_3\le3^{-2\ell}$. Completeness and the contraction
estimate give a unique fixed point $\varphi(t)\in\mathbb Z_3$. It satisfies
\begin{equation}\label{eq:lp:circle-model}
 t^2+2m\varphi(t)+9\varphi(t)^2=0.
\end{equation}
Subtracting \eqref{eq:lp:circle-model} at $t+h$ and $t$ gives
\begin{equation}\label{eq:lp:circle-increment}
 [2m+9(\varphi(t+h)+\varphi(t))]
 (\varphi(t+h)-\varphi(t))=-h(2t+h).
\end{equation}
The bracket is a unit, so
$\varphi(t+h)-\varphi(t)\in h\mathbb Z_3$.
For $u+v=w+z$, let
$D=\varphi(u)+\varphi(v)-\varphi(w)-\varphi(z)$. The identities
\[
 \begin{gathered}
 u^2+v^2-w^2-z^2=2(u-w)(u-z),\\
 \varphi(u)^2+\varphi(v)^2-\varphi(w)^2-\varphi(z)^2
 =D(D+2\varphi(z))-2(\varphi(u)-\varphi(w))(\varphi(v)-\varphi(w))
 \end{gathered}
\]
and \eqref{eq:lp:circle-model} imply
\begin{equation}\label{eq:lp:circle-difference}
 [2m+9(D+2\varphi(z))]D=-2(u-w)(u-z)
 +18(\varphi(u)-\varphi(w))(\varphi(v)-\varphi(w)).
\end{equation}
By \eqref{eq:lp:circle-increment} and $v-w=-(u-z)$, there are
$\rho_1,\rho_2\in\mathbb Z_3$ such that
\[
 \varphi(u)-\varphi(w)=(u-w)\rho_1,\qquad
 \varphi(v)-\varphi(w)=-(u-z)\rho_2.
\]
The right side of \eqref{eq:lp:circle-difference} is therefore
$-2(u-w)(u-z)(1+9\rho_1\rho_2)$. Both the bracket on its left and
$1+9\rho_1\rho_2$ are units. Taking absolute values proves
\eqref{eq:lp:second-difference}.

Let $S_\nu$ be one of the four classes of $S$ and $R_\nu$ its rotation
to $(0,m)$ modulo $3$. By \eqref{eq:lp:circle-conic}, each
$(x,y)\in R_\nu S_\nu$ gives integers $t=x/3$, $s=(y-m)/9$.
Uniqueness in \eqref{eq:lp:circle-model} gives $s=\varphi(t)$. Thus
\begin{equation}\label{eq:lp:circle-graph-points}
 \begin{gathered}
 \widetilde S_\nu
 =\{(x/3,(y-m)/9):(x,y)\in R_\nu S_\nu\}
 \subset\Gamma_\varphi\cap\Z^2,\\
 R_\nu S_\nu=\begin{pmatrix}0\\m\end{pmatrix}+
 \begin{pmatrix}3&0\\0&9\end{pmatrix}\widetilde S_\nu.
 \end{gathered}
\end{equation}
Let $U_\nu=U_{S_\nu,a^{(\nu)}}$, with coefficients inherited from $a$.
Applying Lemma~\ref{lem:lp:graph-weak} to $\widetilde S_\nu$ and using
\eqref{eq:lp:circle-graph-points} and \eqref{eq:lp:affine-distribution}
gives $\|U_\nu\|_{6,\infty}\le C\|a^{(\nu)}\|_2$.
Since $U_{S,a}=\sum_{\nu=1}^4U_\nu$ and
$\{|U_{S,a}|>\alpha\}\subset\bigcup_{\nu=1}^4\{|U_\nu|>\alpha/4\}$,
\begin{equation}\label{eq:lp:circle-weak}
 \begin{aligned}
 \alpha^6|\{|U_{S,a}|>\alpha\}|
 &\le4^6\sum_{\nu=1}^4\|U_\nu\|_{6,\infty}^6
 \le C\sum_{\nu=1}^4\|a^{(\nu)}\|_2^6
 \le C\Bigl(\sum_{\nu=1}^4\|a^{(\nu)}\|_2^2\Bigr)^3
 =C\|a\|_2^6.
 \end{aligned}
\end{equation}
Together with \eqref{eq:app:parseval} and
\eqref{eq:lp:integer-circle}, this proves \eqref{eq:lp:weak}, equivalently
\begin{equation}\label{eq:lp:distribution}
 |\{x\in\T^2:|F(x)|>t\}|
 \le\min\{1,C_{\mathrm w}^6\|F\|_2^6t^{-6}\}\qquad(t>0).
\end{equation}
For $1\le p\le2$, H\"older and $|\T^2|=1$ give
$\|F\|_p\le\|F\|_2$. For $2<p<6$ and $F\ne0$, 
$\|F\|_p^p=p\int_0^\infty t^{p-1}|\{|F|>t\}|\,dt$ and \eqref{eq:lp:distribution} give
\begin{equation}\label{eq:lp:subcritical-integration}
 \frac{\|F\|_p^p}{(C_{\mathrm w}\|F\|_2)^p}
 \le p\int_0^1t^{p-1}\,dt+p\int_1^\infty t^{p-7}\,dt
 =1+\frac p{6-p}=\frac6{6-p}.
\end{equation}
This proves \eqref{eq:lp:subcritical} with
$C_p=C_{\mathrm w}(6/(6-p))^{1/p}$.
\end{proof}

For $F$ in \eqref{eq:app:weighted-sum}, let
$S=\{z\in P_n:a_z\ne0\}$ and $C_0=1+3C_{\mathrm w}^6$. We claim that
\begin{equation}\label{eq:lp:strong-support}
 \|F\|_6\le C_0^{1/6}[\log(2+|S|)]^{1/6}\|F\|_2.
\end{equation}
Taking $F=f_S$ and using \eqref{eq:app:parseval} and \eqref{eq:energy} gives
\begin{equation}\label{eq:log-energy}
 J_3(S)\le C_0|S|^3\log(2+|S|).
\end{equation}
By homogeneity, it suffices to prove the claim \eqref{eq:lp:strong-support} for
$\|F\|_2=1$. Cauchy--Schwarz and \eqref{eq:app:parseval} give
$\|F\|_\infty\le\sqrt{|S|}$, so \eqref{eq:lp:distribution} yields
\begin{equation}\label{eq:lp:strong-from-weak}
 \begin{aligned}
 \|F\|_6^6
 &=6\int_0^{\sqrt{|S|}}t^5|\{|F|>t\}|\,dt
 \le1+6C_{\mathrm w}^6\int_1^{\sqrt{|S|}}\frac{dt}{t}\le C_0\log(2+|S|).
 \end{aligned}
\end{equation}

\subsection{Proof of the graph estimate}\label{subsec:lp:graph-proof}
Let $\varphi:\mathbb Z_3\to\mathbb Z_3$ satisfy
\eqref{eq:lp:second-difference}. For $j\ge0$ and $0\le r<3^j$, define
\begin{equation}\label{eq:lp:caps}
 \Theta_{j,r}=\{(\xi,\eta)\in\Z^2:\xi\equiv r\pmod{3^j},\quad
              \eta\equiv\varphi(\xi)\pmod{3^{2j}}\}.
\end{equation}
\begin{proposition}\label{prop:lp:graph-properties}
The following properties hold.
\begin{enumerate}
\item The map $\varphi$ is $1$-Lipschitz:
\begin{equation}\label{eq:lp:basic-geometry}
 |\varphi(t+h)-\varphi(t)|_3\le|h|_3
 \qquad(t,h\in\mathbb Z_3).
\end{equation}
\item For each $j\ge0$ and $r\bmod3^j$, there are integers $b_{j,r},c_{j,r}$
such that
\begin{equation}\label{eq:lp:cap-linear}
 \begin{gathered}
 \varphi(r+3^ju)\equiv c_{j,r}+b_{j,r}3^ju\pmod{3^{2j}}
 \qquad(u\in\mathbb Z_3),\\
 M_{j,r}=\begin{pmatrix}3^j&0\\b_{j,r}3^j&3^{2j}\end{pmatrix},\qquad
 \Theta_{j,r}=(r,c_{j,r})+M_{j,r}\Z^2.
 \end{gathered}
\end{equation}
Moreover,
\begin{equation}\label{eq:lp:cap-nesting}
 \Theta_{j+1,r+c3^j}\subset\Theta_{j,r}
 \qquad(c=0,1,2).
\end{equation}
\item The map
\begin{equation}\label{eq:lp:normalization}
 \psi(u)=3^{-2j}\bigl(\varphi(r+3^ju)-c_{j,r}-b_{j,r}3^ju\bigr)
 \qquad(u\in\mathbb Z_3)
\end{equation}
takes values in $\mathbb Z_3$ and satisfies
\begin{equation}\label{eq:lp:normalized-identities}
 |\psi(u)+\psi(v)-\psi(w)-\psi(z)|_3=|u-w|_3|u-z|_3
 \quad(u+v=w+z).
\end{equation}
The inverse of $(u,v)\mapsto(r,c_{j,r})+M_{j,r}(u,v)$ sends integer
points $(t,\varphi(t))\in\Theta_{j,r}$ to integer graph points
$(u,\psi(u))$, where $u=(t-r)/3^j$. A class
$t\equiv r+c3^j\pmod{3^{j+1}}$ becomes $u\equiv c\pmod3$.
\item If $k\ge0$ and $\xi_1,\ldots,\xi_4\in\mathbb Z_3$ satisfy
\begin{equation}\label{eq:lp:pair-congruence-hypotheses}
 \begin{gathered}
 \xi_1+\xi_2\equiv\xi_3+\xi_4\pmod{3^k},\\
 \varphi(\xi_1)+\varphi(\xi_2)
 \equiv\varphi(\xi_3)+\varphi(\xi_4)\pmod{3^k},
 \end{gathered}
\end{equation}
then
\begin{equation}\label{eq:lp:pair-congruence}
 (\xi_1-\xi_3)(\xi_1-\xi_4)\equiv0\pmod{3^k}.
\end{equation}
\end{enumerate}
\end{proposition}
\begin{proof}
For (1), fix $t\in\mathbb Z_3$ and a unit $u$, and let
$\delta_\ell=\varphi(t+(\ell+1)u)-\varphi(t+\ell u)$ for $\ell\ge0$.
By \eqref{eq:lp:second-difference},
\[
 |\delta_\ell-\delta_{\ell'}|_3
 =|u|_3|(\ell-\ell')u|_3=|\ell-\ell'|_3.
\]
For each $k\ge0$, the first $3^k$ increments are therefore distinct
modulo $3^k$ and represent all residue classes. Hence
\[
 \varphi(t+3^ku)-\varphi(t)
 =\sum_{\ell=0}^{3^k-1}\delta_\ell
 \equiv\frac{3^k(3^k-1)}2\equiv0\pmod{3^k}.
\]
Every nonzero $h\in\mathbb Z_3$ has the form $3^ku$ with $u$ a unit,
which proves \eqref{eq:lp:basic-geometry}.

For (2), by \eqref{eq:lp:basic-geometry}, we may choose integers
$b_{j,r},c_{j,r}$ such that
\begin{equation}\label{eq:lp:affine-coefficients}
 c_{j,r}\equiv\varphi(r)\pmod{3^{2j}},\qquad
 3^jb_{j,r}\equiv\varphi(r+3^j)-\varphi(r)\pmod{3^{2j}}.
\end{equation}
For $\ell\ge0$, applying \eqref{eq:lp:second-difference} to
$r+3^j(\ell+1),r,r+3^j\ell,r+3^j$ gives
\[
 \big|\varphi(r+3^j(\ell+1))-\varphi(r+3^j\ell)
       -\varphi(r+3^j)+\varphi(r)\big|_3
 =3^{-2j}|\ell|_3\le3^{-2j}.
\]
Summing the increments and using \eqref{eq:lp:affine-coefficients},
\[
 \varphi(r+3^jN)\equiv c_{j,r}+3^jb_{j,r}N\pmod{3^{2j}}
 \qquad(N\in\Z_{\ge0}).
\]
For $u\in\mathbb Z_3$, choose an integer $0\le N<3^j$ with
$N\equiv u\pmod{3^j}$. Then $r+3^ju\equiv r+3^jN\pmod{3^{2j}}$,
so \eqref{eq:lp:basic-geometry} gives
\[
 \varphi(r+3^ju)\equiv\varphi(r+3^jN)
 \equiv c_{j,r}+3^jb_{j,r}u\pmod{3^{2j}}.
\]
Consequently, \eqref{eq:lp:caps} and
$\Z\cap3^{2j}\mathbb Z_3=3^{2j}\Z$ give
\[
 \Theta_{j,r}
 =\{(r+3^ju,c_{j,r}+3^jb_{j,r}u+3^{2j}v):u,v\in\Z\},
\]
which proves \eqref{eq:lp:cap-linear}. Reducing the two defining
congruences in \eqref{eq:lp:caps} gives \eqref{eq:lp:cap-nesting}.

For (3), \eqref{eq:lp:cap-linear} gives $\psi(u)\in\mathbb Z_3$.
When $u+v=w+z$, the constant and linear terms in
\eqref{eq:lp:normalization} cancel, so \eqref{eq:lp:second-difference} yields
\[
 |\psi(u)+\psi(v)-\psi(w)-\psi(z)|_3
 =3^{2j}|3^j(u-w)|_3|3^j(u-z)|_3=|u-w|_3|u-z|_3.
\]
For $(t,s)\in\Theta_{j,r}\cap\Gamma_\varphi$,
\eqref{eq:lp:cap-linear} and \eqref{eq:lp:normalization} give
\[
 u=\frac{t-r}{3^j}\in\Z,\qquad
 \psi(u)=\frac{s-c_{j,r}-b_{j,r}(t-r)}{3^{2j}}\in\Z.
\]
Moreover, $t\equiv r+c3^j\pmod{3^{j+1}}$ is equivalent to
$u\equiv c\pmod3$.

For (4), let $\xi_2'=\xi_3+\xi_4-\xi_1$. The first congruence in
\eqref{eq:lp:pair-congruence-hypotheses} gives
$\xi_2'\equiv\xi_2\pmod{3^k}$. By \eqref{eq:lp:basic-geometry},
$\varphi(\xi_2')\equiv\varphi(\xi_2)\pmod{3^k}$.
Thus the second congruence is unchanged, while
$\xi_1+\xi_2'=\xi_3+\xi_4$. Applying \eqref{eq:lp:second-difference},
\[
 |(\xi_1-\xi_3)(\xi_1-\xi_4)|_3
 =|\varphi(\xi_1)+\varphi(\xi_2')-\varphi(\xi_3)-\varphi(\xi_4)|_3
 \le3^{-k},
\]
which proves \eqref{eq:lp:pair-congruence}.
\end{proof}

We prove the graph estimate by using the following lemma, proved in
Section~\ref{subsec:lp:bilinear-proof}.
\begin{lemma}[Bilinear estimate]\label{lem:lp:bilinear}
Let $\varphi$ satisfy \eqref{eq:lp:second-difference}, and let
$r_1,r_2\in\{0,1,2\}$ with $r_1\ne r_2$. For $i=1,2$, let $S_i$ be a
finite subset of
\[
 \{(t,s)\in\Gamma_\varphi\cap\Z^2:t\equiv r_i\pmod3\}.
\]
If
\begin{equation}\label{eq:lp:bilinear-hypotheses}
 U_i=U_{S_i,a^{(i)}},\qquad \|a^{(i)}\|_2=1,\qquad
 \|U_i\|_{6,\infty}\le K\quad(i=1,2),\qquad K\ge1,
\end{equation}
then $\|U_1U_2\|_{3,\infty}\le CK$, with an absolute constant.
\end{lemma}

\begin{proof}[Proof of Lemma~\ref{lem:lp:graph-weak}]
We apply the bilinear estimate at each scale and use the  martingale estimate (Lemma \ref{lem:projection:martingale}) to complete the proof by a bootstrap argument.
Let $\Lambda_j=3^j\Z\times\Z$ and
$D_j=E_{\Lambda_j}-E_{\Lambda_{j+1}}$ for $j\ge0$.
We may assume $S\ne\varnothing$. Let
\begin{equation}\label{eq:lp:fixed-constant}
 K_S=\sup_{a\ne0}\frac{\|U_{S,a}\|_{6,\infty}}{\|a\|_2},
 \qquad 1\le K_S\le\sqrt{|S|}<\infty.
\end{equation}
The bounds follow from a single nonzero coefficient and
$\|U_{S,a}\|_{6,\infty}\le\|U_{S,a}\|_\infty\le\sqrt{|S|}\|a\|_2$.
Fix $\|a\|_2=1$ and let $U=U_{S,a}$.
Choose $J\ge2$ so that the finitely many distinct first coordinates of
$S$ remain distinct modulo $3^J$. For $0\le j\le J$ and $r\bmod3^j$, let
\begin{equation}\label{eq:lp:tree-masses}
 U_{j,r}=\sum_{\substack{\xi=(t,s)\in S\\t\equiv r\pmod{3^j}}}
       a_\xi e^{2\pi i\xi\cdot x},\qquad
 A_{j,r}=\sum_{\substack{\xi=(t,s)\in S\\t\equiv r\pmod{3^j}}}|a_\xi|^2.
\end{equation}
Then
\begin{equation}\label{eq:lp:tree-identities}
 U_{j,r}=\sum_{c=0}^2U_{j+1,r+c3^j},\qquad
 A_{j,r}=\sum_{c=0}^2A_{j+1,r+c3^j}\quad(j<J),\qquad
 \sum_{r\bmod3^j}A_{j,r}=1.
\end{equation}
By \eqref{eq:projection:identity},
\begin{equation}\label{eq:lp:class-square}
 E_{\Lambda_j}|U|^2
 =\sum_{\substack{\xi=(t,s),\,\eta=(t',s')\in S\\t\equiv t'\pmod{3^j}}}
 a_\xi\overline{a_\eta}e^{2\pi i(\xi-\eta)\cdot x}
 =\sum_{r\bmod3^j}|U_{j,r}|^2.
\end{equation}
The choice of $J$ gives
\begin{equation}\label{eq:lp:terminal-square}
 E_{\Lambda_J}|U|^2=\sum_{\xi\in S}|a_\xi|^2=1.
\end{equation}
For $j<J$, \eqref{eq:projection:properties} gives
$D_jE_{\Lambda_j}=E_{\Lambda_j}^2-E_{\Lambda_{j+1}}E_{\Lambda_j}=D_j$.
Thus \eqref{eq:lp:class-square} implies
\begin{equation}\label{eq:lp:linear-decomposition}
 D_j|U|^2=D_jE_{\Lambda_j}|U|^2
 =D_j\sum_{r\bmod3^j}|U_{j,r}|^2
 =\sum_{r\bmod3^j}D_j|U_{j,r}|^2.
\end{equation}
Applying the same projections to $|U_{j,r}|^2$ and using
\eqref{eq:lp:tree-identities},
\begin{equation}\label{eq:lp:child-cross-terms}
 \begin{aligned}
 D_j|U_{j,r}|^2
 &=\Bigl|\sum_{c=0}^2U_{j+1,r+c3^j}\Bigr|^2
   -\sum_{c=0}^2|U_{j+1,r+c3^j}|^2\\
 &=\sum_{\substack{0\le c,d\le2\\c\ne d}}
   U_{j+1,r+c3^j}\overline{U_{j+1,r+d3^j}}.
 \end{aligned}
\end{equation}
By \eqref{eq:lp:tree-masses} and \eqref{eq:lp:caps},
$\operatorname{supp}\widehat U_{j,r}\subset\Theta_{j,r}$.
Thus \eqref{eq:lp:cap-linear} gives
\[
 \operatorname{supp}\widehat{|U_{j,r}|^2}\subset M_{j,r}\Z^2,
 \qquad
 D_jE_{M_{j,r}\Z^2}(D_j|U_{j,r}|^2)=D_j|U_{j,r}|^2.
\]
By \eqref{eq:projection:properties}, the projection $D_jE_{M_{j,r}\Z^2}$
has $L^\infty$ operator norm at most $2$ and frequency support
\begin{equation}\label{eq:lp:cross-support}
 \{(\xi,\eta)\in M_{j,r}\Z^2:v_3(\xi)=j\}.
\end{equation}
For each fixed $j<J$, these supports are pairwise disjoint over
$r\bmod3^j$. Indeed, a common frequency for $r\not\equiv r'\pmod{3^j}$
would satisfy, by \eqref{eq:lp:cap-linear},
\[
 \varphi(r+\xi)-\varphi(r)\equiv b_{j,r}\xi\equiv\eta
 \equiv\varphi(r'+\xi)-\varphi(r')\pmod{3^{2j}}.
\]
Since $v_3(\xi)=j$, \eqref{eq:lp:second-difference} would give
\[
 \begin{aligned}
 3^{-2j}&\ge
 |\varphi(r+\xi)+\varphi(r')-\varphi(r)-\varphi(r'+\xi)|_3\\
 &=3^{-j}|r-r'|_3>3^{-2j},
 \end{aligned}
\]
a contradiction. Thus \eqref{eq:projection:sum:weak} applies to these projections.

Fix $j,r$ and $c\ne d$ in $\{0,1,2\}$ with
$A_{j+1,r+c3^j}A_{j+1,r+d3^j}>0$. For $e\in\{c,d\}$,
\eqref{eq:lp:tree-masses} and \eqref{eq:lp:fixed-constant} give
\[
 \sum_{\substack{\xi=(t,s)\in S\\t\equiv r+e3^j\pmod{3^{j+1}}}}
       \frac{|a_\xi|^2}{A_{j+1,r+e3^j}}=1,
 \qquad
 \frac{\|U_{j+1,r+e3^j}\|_{6,\infty}}{\sqrt{A_{j+1,r+e3^j}}}\le K_S.
\]
Proposition~\ref{prop:lp:graph-properties}(3) sends these frequencies to
integer points $(u,\psi(u))$ with
\[
 u=\frac{t-r}{3^j}\in e+3\Z,\qquad
 \psi(u)=\frac{s-c_{j,r}-b_{j,r}(t-r)}{3^{2j}}\in\Z.
\]
Apply Lemma~\ref{lem:lp:bilinear} to the normalized sums in these
coordinates. Equations~\eqref{eq:lp:affine-invariance} and \eqref{eq:scaling}
preserve their weak norms and the product's weak norm, giving
\begin{equation}\label{eq:lp:pair-bound}
 \|U_{j+1,r+c3^j}\overline{U_{j+1,r+d3^j}}\|_{3,\infty}
 \le CK_S\sqrt{A_{j+1,r+c3^j}A_{j+1,r+d3^j}}.
\end{equation}

The inclusion
$\{|\sum_{\ell=1}^6f_\ell|>t\}\subset\bigcup_{\ell=1}^6\{|f_\ell|>t/6\}$
gives
\begin{equation}\label{eq:lp:six-term-bound}
 \left\|\sum_{\ell=1}^6f_\ell\right\|_{3,\infty}
 \le6\Bigl(\sum_{\ell=1}^6\|f_\ell\|_{3,\infty}^3\Bigr)^{1/3}
 \le6\sum_{\ell=1}^6\|f_\ell\|_{3,\infty}.
\end{equation}
For $x_0,x_1,x_2\ge0$ and $s=x_0+x_1+x_2$, Cauchy--Schwarz and
$s-x_c\le2\sqrt{s}(\sqrt{s}-\sqrt{x_c})$ give
\begin{equation}\label{eq:lp:mass-comparison}
 \Bigl(\sum_{\substack{0\le c,d\le2\\c\ne d}}\sqrt{x_cx_d}\Bigr)^2
 \le6\sum_{c=0}^2x_c(s-x_c)
 \le12\sqrt{s}\Bigl(s^{3/2}-\sum_{c=0}^2x_c^{3/2}\Bigr).
\end{equation}
Also, \eqref{eq:lp:tree-identities} gives
\begin{equation}\label{eq:lp:mass-decrements}
 A_{j,r}^{3/2}\ge\sum_{c=0}^2A_{j+1,r+c3^j}^{3/2},\qquad
 \sum_{r\bmod3^j}A_{j,r}^{3/2}\le1.
\end{equation}
Combining \eqref{eq:lp:child-cross-terms} with
\eqref{eq:lp:pair-bound}--\eqref{eq:lp:mass-comparison}, we obtain
\begin{equation}\label{eq:lp:parent-bound}
 \|D_j|U_{j,r}|^2\|_{3,\infty}^{3/2}
 \le CK_S^{3/2}A_{j,r}^{3/8}
 \Bigl(A_{j,r}^{3/2}-\sum_{c=0}^2A_{j+1,r+c3^j}^{3/2}\Bigr)^{3/4}.
\end{equation}
By \eqref{eq:lp:linear-decomposition}, \eqref{eq:projection:sum:weak},
\eqref{eq:lp:parent-bound}, H\"older with exponents $4$ and $4/3$, and
\eqref{eq:lp:mass-decrements},
\begin{equation}\label{eq:lp:layer-summation}
 \begin{aligned}
 \|D_j|U|^2\|_{3,\infty}^2
 &\le C\Bigl(\sum_{r\bmod3^j}
       \|D_j|U_{j,r}|^2\|_{3,\infty}^{3/2}\Bigr)^{4/3}\\
 &\le CK_S^2\bigl[\sum_{r\bmod3^j}A_{j,r}^{3/8}
       \Bigl(A_{j,r}^{3/2}-\sum_{c=0}^2A_{j+1,r+c3^j}^{3/2}\Bigr)^{3/4}
       \bigr]^{4/3}\\
 &\le CK_S^2\Bigl(\sum_{r\bmod3^j}A_{j,r}^{3/2}\Bigr)^{1/3}
       \Bigl(\sum_{r\bmod3^j}A_{j,r}^{3/2}
                    -\sum_{r\bmod3^{j+1}}A_{j+1,r}^{3/2}\Bigr)\\
 &\le CK_S^2\Bigl(\sum_{r\bmod3^j}A_{j,r}^{3/2}
                    -\sum_{r\bmod3^{j+1}}A_{j+1,r}^{3/2}\Bigr).
 \end{aligned}
\end{equation}
Since $A_{0,0}=1$ and $\sum_{r\bmod3^J}A_{J,r}^{3/2}
=\sum_{\xi\in S}|a_\xi|^3$ by \eqref{eq:lp:tree-masses}, summing
\eqref{eq:lp:layer-summation} and applying Lemma~\ref{lem:projection:martingale}
to $|U|^2$ and $\Lambda_0\supset\cdots\supset\Lambda_J$ gives
\begin{equation}\label{eq:lp:mass-telescoping}
 \begin{aligned}
 \|U\|_{6,\infty}^4
 &\le C+C\sum_{j=0}^{J-1}\|D_j|U|^2\|_{3,\infty}^2\\
 &\le C+CK_S^2\Bigl(1-\sum_{\xi\in S}|a_\xi|^3\Bigr)
 \le C(1+K_S^2).
 \end{aligned}
\end{equation}
Here we used \eqref{eq:lp:terminal-square} and
$\||U|^2\|_{3,\infty}=\|U\|_{6,\infty}^2$.
Taking the supremum over $\|a\|_2=1$ gives
$K_S^4\le C(1+K_S^2)\le2CK_S^2$. Since $1\le K_S<\infty$,
$K_S\le\sqrt{2C}$.
\end{proof}

\subsection{Proof of the bilinear estimate}\label{subsec:lp:bilinear-proof}
For $j\ge0$, write $Q_j=E_{3^j\Z^2}$. The following proposition records the properties needed for the truncation.
All identities for bounded functions are understood almost everywhere.
\begin{proposition}\label{prop:lp:bounded-functions}
\begin{enumerate}
\item Let $\Lambda\subset\Z^2$ be a full-rank sublattice and let
$V\in L^\infty(\T^2)$ satisfy
$\operatorname{supp}\widehat V\subset\zeta+\Lambda$, with $\zeta\in\Z^2$.
Then
\begin{equation}\label{eq:lp:cap-translations}
 V(x+y)=e^{2\pi i\zeta\cdot y}V(x)
 \qquad(y\in\Lambda^\perp).
\end{equation}
For $L>0$, both
$V\mathbf1_{\{|V|\le L\}}$ and $V\mathbf1_{\{|V|>L\}}$ have Fourier
support in $\zeta+\Lambda$. Moreover,
\begin{equation}\label{eq:lp:modulus-invariance}
 E_\Lambda|V|^2=|V|^2,\qquad
 E_\Lambda\bigl(|V|^2\mathbf1_{\{|V|>L\}}\bigr)
       =|V|^2\mathbf1_{\{|V|>L\}}.
\end{equation}
In particular, if $\operatorname{supp}\widehat V\subset\Theta_{j,r}$,
these identities hold with $E_\Lambda$ replaced by $Q_j$.
\item If $V_r\in L^\infty(\T^2)$ and
$\operatorname{supp}\widehat V_r\subset\Theta_{j,r}$ for every
$r\bmod3^j$, then
\begin{equation}\label{eq:lp:four-cap}
 \left\|\sum_{r\bmod3^j}V_r\right\|_4^4
 =2\left\|\sum_{r\bmod3^j}|V_r|^2\right\|_2^2
       -\sum_{r\bmod3^j}\|V_r\|_4^4.
\end{equation}
\end{enumerate}
\end{proposition}
\begin{proof}
For (1), if $y\in\Lambda^\perp$, the support assumption gives
\[
 \widehat{V(\cdot+y)}(\nu)=e^{2\pi i\nu\cdot y}\widehat V(\nu)
 =e^{2\pi i\zeta\cdot y}\widehat V(\nu)\qquad(\nu\in\Z^2).
\]
Thus \eqref{eq:lp:cap-translations} holds in $L^2$, and
$|V(x+y)|=|V(x)|$. The support assumption and
\eqref{eq:projection:module} therefore give
\[
 E_\Lambda\bigl(e^{-2\pi i\zeta\cdot x}V\mathbf1_{\{|V|\le L\}}\bigr)
 =e^{-2\pi i\zeta\cdot x}V\mathbf1_{\{|V|\le L\}}.
\]
Hence $V\mathbf1_{\{|V|\le L\}}$ has Fourier support in $\zeta+\Lambda$.
The same holds for $V\mathbf1_{\{|V|>L\}}=V-V\mathbf1_{\{|V|\le L\}}$.
The functions $|V|^2$ and $|V|^2\mathbf1_{\{|V|>L\}}$ are invariant
under these translations, so \eqref{eq:projection:identity} proves
\eqref{eq:lp:modulus-invariance}. For support in $\Theta_{j,r}$,
\eqref{eq:lp:cap-linear} gives $M_{j,r}\Z^2\subset3^j\Z^2$, hence
$Q_jE_{M_{j,r}\Z^2}=E_{M_{j,r}\Z^2}$, which proves the assertions for $Q_j$.

For (2), first let the $V_r$ be trigonometric polynomials.
If $\int V_{r_1}V_{r_2}\overline{V_{r_3}V_{r_4}}\ne0$, Fourier
orthogonality gives $(\xi_\ell,\eta_\ell)\in\Theta_{j,r_\ell}$, $1\le\ell\le4$,
with
\[
 \xi_1+\xi_2=\xi_3+\xi_4,\qquad \eta_1+\eta_2=\eta_3+\eta_4.
\]
By \eqref{eq:lp:caps} and \eqref{eq:lp:second-difference},
\[
 |\xi_1-\xi_3|_3|\xi_1-\xi_4|_3
 =|\varphi(\xi_1)+\varphi(\xi_2)-\varphi(\xi_3)-\varphi(\xi_4)|_3
 \le3^{-2j}.
\]
Thus $\max\{v_3(\xi_1-\xi_3),v_3(\xi_1-\xi_4)\}\ge j$.
Since $\xi_2-\xi_4=-(\xi_1-\xi_3)$ and
$\xi_2-\xi_3=-(\xi_1-\xi_4)$,
\begin{equation}\label{eq:lp:four-class-pairing}
 \int V_{r_1}V_{r_2}\overline{V_{r_3}V_{r_4}}=0
 \quad\text{unless}\quad
 (r_3,r_4)\in\{(r_1,r_2),(r_2,r_1)\},
\end{equation}
where indices are taken modulo $3^j$.
The two pairings overlap only when all four indices agree. Expanding
and subtracting this overlap gives
\[
 \begin{aligned}
 \left\|\sum_rV_r\right\|_4^4
 &=2\sum_{r,s}\int|V_r|^2|V_s|^2-\sum_r\int|V_r|^4&=2\left\|\sum_r|V_r|^2\right\|_2^2-\sum_r\|V_r\|_4^4.
 \end{aligned}
\]
For bounded $V_r$, Fej\'er means on $\T^2$ preserve the Fourier supports
and converge in $L^4$. H\"older then gives \eqref{eq:lp:four-cap} by passage to the limit.
\end{proof}

The following truncation lemma is motivated by \cite[Proposition~4.1]{DFGGL}.
\begin{lemma}\label{lem:lp:truncation}
Let $S\subset\Gamma_\varphi\cap\Z^2$ be finite, where $\varphi$ satisfies
\eqref{eq:lp:second-difference}, let $a=(a_\xi)_{\xi\in S}$, and let
$A=\|a\|_2^2$. Choose $J\ge2$ so that the first coordinates of $S$ are
distinct modulo $3^J$. For $L>0$, define
\begin{equation}\label{eq:lp:truncation-leaves}
 h_{J,r}(x)=\sum_{\substack{\xi=(t,s)\in S\\t\equiv r\pmod{3^J}}}
 a_\xi e^{2\pi i\xi\cdot x}\qquad(r\bmod3^J),
\end{equation}
and, successively for $j=J-1,J-2,\ldots,1$, define
\begin{equation}\label{eq:lp:truncation-recursion}
 h_{j,r}=\sum_{c=0}^2h_{j+1,r+c3^j}
                  \mathbf1_{\{|h_{j+1,r+c3^j}|\le L\}}
 \qquad(r\bmod3^j).
\end{equation}
For $1\le j\le J$, let
\begin{equation}\label{eq:lp:truncation-square}
 G_j=\sum_{r\bmod3^j}|h_{j,r}|^2
     +\sum_{k=j+1}^J\sum_{r\bmod3^k}
           |h_{k,r}|^2\mathbf1_{\{|h_{k,r}|>L\}}.
\end{equation}
Then
\begin{equation}\label{eq:lp:truncation-estimate}
 G_J=A,\qquad Q_jG_1=G_j\quad(1\le j\le J),\qquad
 \|G_1-A\|_2^2\le\tfrac92L^2A.
\end{equation}
\end{lemma}
\begin{proof}
By \eqref{eq:lp:truncation-leaves},
$\operatorname{supp}\widehat h_{J,r}\subset\Theta_{J,r}$.
Proposition~\ref{prop:lp:bounded-functions}(1) preserves the support under
truncation, and \eqref{eq:lp:cap-nesting} places each child in its parent
class. Thus downward induction in \eqref{eq:lp:truncation-recursion} gives
\begin{equation}\label{eq:lp:truncation-support}
 \begin{gathered}
 \operatorname{supp}\widehat h_{j,r}\subset\Theta_{j,r}
 \qquad(1\le j\le J),\\
 |h_{j,r}|\le\sum_{c=0}^2
 |h_{j+1,r+c3^j}|\mathbf1_{\{|h_{j+1,r+c3^j}|\le L\}}
 \le3L\qquad(1\le j<J).
 \end{gathered}
\end{equation}
Let
\begin{equation}\label{eq:lp:truncated-children}
 w_{j+1,r}=h_{j+1,r}\mathbf1_{\{|h_{j+1,r}|\le L\}}
 \qquad(1\le j<J).
\end{equation}
Proposition~\ref{prop:lp:bounded-functions}(1) gives
$\operatorname{supp}\widehat w_{j+1,r}\subset\Theta_{j+1,r}$ and
$Q_{j+1}|w_{j+1,r}|^2=|w_{j+1,r}|^2$.
For $c\ne d$, the first Fourier coordinate of
$w_{j+1,r+c3^j}\overline{w_{j+1,r+d3^j}}$ is congruent to
$(c-d)3^j\not\equiv0\pmod{3^{j+1}}$, so $Q_{j+1}$ annihilates this product.
Also, Proposition~\ref{prop:lp:bounded-functions}(1) and
\eqref{eq:projection:properties} imply that $Q_j$ fixes $|h_{k,r}|^2$ and $|h_{k,r}|^2\mathbf1_{\{|h_{k,r}|>L\}}$ for $k\ge j$.
Expanding \eqref{eq:lp:truncation-recursion} and then using
\eqref{eq:lp:truncation-square}, we obtain
\begin{equation}\label{eq:lp:truncation-averages}
 \begin{gathered}
 Q_{j+1}|h_{j,r}|^2=\sum_{c=0}^2|w_{j+1,r+c3^j}|^2,\qquad Q_jG_j=G_j,\\
 \begin{aligned}
 Q_{j+1}G_j
 &=\sum_{r\bmod3^{j+1}}|w_{j+1,r}|^2
   +\sum_{k=j+1}^J\sum_{r\bmod3^k}
       |h_{k,r}|^2\mathbf1_{\{|h_{k,r}|>L\}}\\
 &=G_{j+1}\qquad(1\le j<J).
 \end{aligned}
 \end{gathered}
\end{equation}
The last equality uses
$|w_{j+1,r}|^2+|h_{j+1,r}|^2\mathbf1_{\{|h_{j+1,r}|>L\}}
=|h_{j+1,r}|^2$.
Each sum in \eqref{eq:lp:truncation-leaves} has at most one term, so
$G_J=A$. Since $Q_{j+1}Q_j=Q_{j+1}$, \eqref{eq:lp:truncation-averages}
gives $Q_jG_1=G_j$ by induction. Preservation of integrals then yields
\begin{equation}\label{eq:lp:truncation-integrals}
 \int G_j=\int G_1=\int G_J=A\qquad(1\le j\le J).
\end{equation}
For $1\le j<J$, let
\begin{equation}\label{eq:lp:truncation-components}
 \begin{gathered}
 H_j=\sum_{r\bmod3^j}|h_{j,r}|^2,\qquad
 T_j=\sum_{r\bmod3^{j+1}}|w_{j+1,r}|^2,\\
 R_j=\sum_{k=j+1}^J\sum_{r\bmod3^k}
           |h_{k,r}|^2\mathbf1_{\{|h_{k,r}|>L\}},\qquad
 G_j=H_j+R_j,\quad G_{j+1}=T_j+R_j.
 \end{gathered}
\end{equation}
By \eqref{eq:lp:truncation-averages},
$Q_{j+1}H_j=T_j$, $Q_{j+1}T_j=T_j$, and $Q_{j+1}R_j=R_j$.
Expanding the squares and using \eqref{eq:projection:module} gives
\begin{equation}\label{eq:lp:truncation-orthogonality}
 \begin{aligned}
 \|G_j\|_2^2-\|G_{j+1}\|_2^2
 &=\|H_j\|_2^2-\|T_j\|_2^2+2\int R_jQ_{j+1}(H_j-T_j)\\
 &=\|H_j\|_2^2-\|T_j\|_2^2.
 \end{aligned}
\end{equation}
On the other hand, \eqref{eq:lp:truncation-recursion} gives
\[
 \sum_{r\bmod3^j}h_{j,r}
 =\sum_{r\bmod3^j}\sum_{c=0}^2w_{j+1,r+c3^j}
 =\sum_{r\bmod3^{j+1}}w_{j+1,r}.
\]
Applying \eqref{eq:lp:four-cap} at these two levels yields
\[
 2\|H_j\|_2^2-\sum_{r\bmod3^j}\|h_{j,r}\|_4^4
 =2\|T_j\|_2^2-\sum_{r\bmod3^{j+1}}\|w_{j+1,r}\|_4^4.
\]
Together with \eqref{eq:lp:truncation-orthogonality} and
\eqref{eq:lp:truncated-children}, this gives
\begin{equation}\label{eq:lp:truncation-quartic}
 2\bigl(\|G_j\|_2^2-\|G_{j+1}\|_2^2\bigr)
 =\sum_{r\bmod3^j}\|h_{j,r}\|_4^4
 -\sum_{r\bmod3^{j+1}}\int_{\{|h_{j+1,r}|\le L\}}|h_{j+1,r}|^4.
\end{equation}
Summing over $1\le j<J$ and using $G_J=A$ and
\eqref{eq:lp:truncation-integrals}, we conclude that
\begin{equation}\label{eq:lp:truncation-telescoping}
 \begin{aligned}
 2\|G_1-A\|_2^2
 &=\sum_{r\bmod3}\|h_{1,r}\|_4^4
   +\sum_{j=2}^{J-1}\sum_{r\bmod3^j}
        \int_{\{|h_{j,r}|>L\}}|h_{j,r}|^4-\sum_{r\bmod3^J}\int_{\{|h_{J,r}|\le L\}}|h_{J,r}|^4\\
 &\le9L^2\int G_1=9L^2A.
 \end{aligned}
\end{equation}
The inequality discards the nonpositive final term and uses
$|h_{j,r}|^4\le9L^2|h_{j,r}|^2$ for $j<J$, by
\eqref{eq:lp:truncation-support}, and \eqref{eq:lp:truncation-square}.
\end{proof}

\begin{proof}[Proof of Lemma~\ref{lem:lp:bilinear}]
For $\alpha>0$, let
\begin{equation}\label{eq:lp:balanced-level}
 E_\alpha=\{|U_1U_2|>\alpha^2,\quad\max_i|U_i|\le K\alpha\}.
\end{equation}
Since $\{|U_1U_2|>\alpha^2\}\setminus E_\alpha
\subset\bigcup_{i=1}^2\{|U_i|>K\alpha\}$,
\eqref{eq:lp:bilinear-hypotheses} gives
\begin{equation}\label{eq:lp:bilinear-level-reduction}
 \alpha^6|\{|U_1U_2|>\alpha^2\}|
 \le\alpha^6|E_\alpha|+\sum_{i=1}^2K^{-6}\|U_i\|_{6,\infty}^6
 \le\alpha^6|E_\alpha|+2.
\end{equation}
The left side is at most $4^6$ when $\alpha\le4$.
Henceforth let $\alpha>4$ and $L=16K/\alpha$.
Apply Lemma~\ref{lem:lp:truncation} to $U_1+U_2$. Since
$S_1\cap S_2=\varnothing$, we have $A=2$ and $G_J=2$.
Let
\begin{align}
 \Omega_j&=\{G_j>4\}\cap\bigcap_{k=j+1}^J\{G_k\le4\}
 \quad(1\le j<J),\label{eq:lp:stopping-sets}\\
 V_{i,j}&=\sum_{\substack{r\bmod3^j\\r\equiv r_i\pmod3}}h_{j,r}
 \quad(i=1,2,\ 1\le j\le J).\label{eq:lp:stopped-sums}
\end{align}
The sets $\Omega_j$ are pairwise disjoint. For fixed $1\le j<J$,
\eqref{eq:projection:properties} and \eqref{eq:lp:truncation-estimate} give
$Q_jG_k=Q_jQ_kG_1=G_k$ for every $j\le k\le J$.
Translations by $(3^j\Z^2)^\perp$ permute the summands in
\eqref{eq:projection:identity}, so each such $G_k$ is invariant under
these translations. By \eqref{eq:lp:stopping-sets}, the same holds for
$\mathbf1_{\Omega_j}$, and averaging gives
$Q_j\mathbf1_{\Omega_j}=\mathbf1_{\Omega_j}$.

By \eqref{eq:lp:truncation-recursion} and \eqref{eq:lp:stopped-sums},
\[
 V_{i,j+1}-V_{i,j}
 =\sum_{\substack{r\bmod3^{j+1}\\r\equiv r_i\pmod3}}
 h_{j+1,r}\mathbf1_{\{|h_{j+1,r}|>L\}}.
\]
Since $V_{i,J}=U_i$, summing and using
$|h|\mathbf1_{\{|h|>L\}}\le L^{-1}|h|^2\mathbf1_{\{|h|>L\}}$
and \eqref{eq:lp:truncation-square} yields
\begin{equation}\label{eq:lp:stopped-error}
 \begin{aligned}
 U_i-V_{i,j}
 &=\sum_{k=j+1}^J\sum_{\substack{r\bmod3^k\\r\equiv r_i\pmod3}}
       h_{k,r}\mathbf1_{\{|h_{k,r}|>L\}},\\
 |U_i-V_{i,j}|
 &\le\frac1L\sum_{k=j+1}^J
       \sum_{\substack{r\bmod3^k\\r\equiv r_i\pmod3}}
       |h_{k,r}|^2\mathbf1_{\{|h_{k,r}|>L\}}
 \le\frac{G_{j+1}}L.
 \end{aligned}
\end{equation}
On $E_\alpha$, \eqref{eq:lp:balanced-level} implies $|U_i|>\alpha/K$.
Thus, on $E_\alpha\cap\{G_{j+1}\le4\}$,
\begin{equation}\label{eq:lp:stopped-product}
 |V_{i,j}|\ge|U_i|-\frac\alpha{4K}\ge\tfrac34|U_i|,
 \qquad |V_{1,j}V_{2,j}|>\tfrac9{16}\alpha^2.
\end{equation}
If $G_\ell(x)>4$ for some $1\le\ell\le J$, let $k_*$ be the largest
such index. Since $G_J=2$, we have $k_*<J$ and $x\in\Omega_{k_*}$.
Hence $G_\ell(x)\le4$ for every $1\le\ell\le J$ whenever
$x\notin\bigcup_{j=1}^{J-1}\Omega_j$. Since $V_{i,1}=h_{1,r_i}$,
\eqref{eq:lp:truncation-square} and \eqref{eq:lp:stopped-error} give
\begin{equation}\label{eq:lp:stopping-cover}
 \max_i|U_i|\le\sqrt{G_1}+\frac4L
 \le2+\frac\alpha{4K}<\alpha,
 \qquad E_\alpha\subset\bigcup_{j=1}^{J-1}\Omega_j.
\end{equation}

Fix $1\le j<J$. By \eqref{eq:lp:truncation-support} and
\eqref{eq:lp:cap-linear},
\[
 \operatorname{supp}\widehat h_{j,r}
 \subset(r,c_{j,r})+3^j\Z^2,\qquad
 c_{j,r}\equiv\varphi(r)\pmod{3^j}.
\]
Let $r,s,u,v$ be residue classes modulo $3^j$, with
$r,u\equiv r_1\pmod3$ and $s,v\equiv r_2\pmod3$.
By \eqref{eq:lp:cap-translations} and \eqref{eq:projection:identity},
$Q_j$ annihilates $h_{j,r}h_{j,s}\overline{h_{j,u}h_{j,v}}$ unless
\[
 r+s\equiv u+v\pmod{3^j},\qquad
 \varphi(r)+\varphi(s)\equiv\varphi(u)+\varphi(v)\pmod{3^j}.
\]
Since $r_1\ne r_2$, \eqref{eq:lp:pair-congruence} and the first congruence give
\[
 (r-u)(r-v)\equiv0\pmod{3^j},\quad |r-v|_3=1
 \quad\Longrightarrow\quad r\equiv u,\ s\equiv v\pmod{3^j}.
\]
The diagonal product is fixed by $Q_j$, by
\eqref{eq:lp:modulus-invariance} and \eqref{eq:projection:module}. Therefore
\begin{equation}\label{eq:lp:four-product-projection}
 Q_j(h_{j,r}h_{j,s}\overline{h_{j,u}h_{j,v}})
 =\begin{cases}
 |h_{j,r}|^2|h_{j,s}|^2,&r=u,\ s=v,\\
 0,&\text{otherwise},
 \end{cases}
\end{equation}
where the indices are taken modulo $3^j$.
Expanding \eqref{eq:lp:stopped-sums} and using $4ab\le(a+b)^2$ gives
\begin{equation}\label{eq:lp:stopped-local-bound}
 \begin{aligned}
 Q_j|V_{1,j}V_{2,j}|^2
 &=\sum_{\substack{r,s\bmod3^j\\r\equiv r_1,\ s\equiv r_2\pmod3}}
       |h_{j,r}|^2|h_{j,s}|^2
 =\prod_{i=1}^2\Bigl(
       \sum_{\substack{r\bmod3^j\\r\equiv r_i\pmod3}}|h_{j,r}|^2\Bigr)\\
 &\le\frac14\Bigl(\sum_{r\bmod3^j}|h_{j,r}|^2\Bigr)^2
 \le\frac14G_j^2,
 \end{aligned}
\end{equation}
where the last inequality follows from \eqref{eq:lp:truncation-square}.
On $\Omega_j$, $G_j>4$, so \eqref{eq:lp:truncation-estimate} and Jensen give
\[
 G_j^2\le4(G_j-2)^2
 =4|Q_j(G_1-2)|^2\le4Q_j|G_1-2|^2.
\]
Since $Q_j\mathbf1_{\Omega_j}=\mathbf1_{\Omega_j}$,
\eqref{eq:projection:module} gives $\int_{\Omega_j}Q_jf=\int_{\Omega_j}f$.
Using \eqref{eq:lp:stopped-product}, \eqref{eq:lp:stopping-cover},
\eqref{eq:lp:stopped-local-bound}, disjointness, and
\eqref{eq:lp:truncation-estimate} with $A=2$, we obtain
\begin{equation}\label{eq:lp:stopping-summation}
 \begin{aligned}
 \alpha^4|E_\alpha|
 &\le C\sum_{j=1}^{J-1}\int_{\Omega_j}|V_{1,j}V_{2,j}|^2
   =C\sum_{j=1}^{J-1}\int_{\Omega_j}Q_j|V_{1,j}V_{2,j}|^2\\
 &\le C\sum_{j=1}^{J-1}\int_{\Omega_j}G_j^2
 \le C\sum_{j=1}^{J-1}\int_{\Omega_j}Q_j|G_1-2|^2\\
 &=C\sum_{j=1}^{J-1}\int_{\Omega_j}|G_1-2|^2
 \le C\|G_1-2\|_2^2\le CL^2.
 \end{aligned}
\end{equation}
Since $L=16K/\alpha$, \eqref{eq:lp:bilinear-level-reduction} and
\eqref{eq:lp:stopping-summation}, together with the case $\alpha\le4$, yield
\[
 \|U_1U_2\|_{3,\infty}^3
 =\sup_{\alpha>0}\alpha^6|\{|U_1U_2|>\alpha^2\}|
 \le CK^2.
\]
This proves the lemma.
\end{proof}

\section{Proof of Theorem~\ref{thm:main}}\label{sec:closure}
We prove the cubic energy bound using four propositions. 

\begin{proposition}\label{prop:recurrence}
Let $p=\varpi\bar\varpi$ be a split rational prime, and let
$m,\nu\ge1$ be integers with $p\nmid m$. Let
\begin{equation}\label{eq:split:R}
 R_p(m)=\|E_{p\Z^2}f_{P_m}^{\,2}\|_3.
\end{equation}
There is an absolute constant $C_1\ge1$ such that
\begin{equation}\label{eq:recurrence}
 \begin{aligned}
 J_3(P_{p^\nu m})&\le J_3(P_{p^{\nu-1}m})+2J_3(P_m)+C_1R_p(m)J_3(P_{p^{\nu-1}m})^{2/3}.
 \end{aligned}
\end{equation}
\end{proposition}

\begin{proposition}\label{prop:local}
Let $p$ be an odd prime and $m\ge1$ an integer such that $p\nmid m$
and $P_m\ne\varnothing$. Write $\mathbb F_p=\Z/p\Z$, and let
\begin{equation}\label{eq:local:classes}
 \begin{gathered}
 S_c=\{z\in P_m:z\equiv c\pmod p\}\qquad(c\in\mathbb F_p^2),\qquad
 \rho_p(P_m)=\max_{c\in\mathbb F_p^2}\frac{|S_c|}{M_m}.
 \end{gathered}
\end{equation}
Then the quantity in \eqref{eq:split:R} satisfies
\begin{equation}\label{eq:local}
 R_p(m)^3
 \le M_m^3\Bigl(\tfrac43+16C_0\rho_p(P_m)\log(2M_m)\Bigr).
\end{equation}
\end{proposition}

\begin{proposition}\label{prop:residues}
Let $n=\prod_{j=1}^r p_j^{\nu_j}$, where $r\ge1$,
$p_1<\cdots<p_r$ are distinct split rational primes, and the $\nu_j$ are
positive integers.
If $p=p_r$ and $m=n/p^{\nu_r}$, then the quantity in
\eqref{eq:local:classes} satisfies
\begin{equation}\label{eq:rho}
 \rho_p(P_m)\le\frac1{r+1}.
\end{equation}
\end{proposition}

\begin{proposition}\label{prop:unit}
There is an absolute constant $C_*>0$ such that, if $n\ge1$ has
$r$ distinct prime factors, all of which are split, then
\begin{equation}\label{eq:ESS}
 \frac{J_3(P_n)}{M_n^3}
 \le15+\exp\bigl(C_*(r+1)-\log M_n\bigr).
\end{equation}
\end{proposition}

\begin{proof}[Proof of Theorem~\ref{thm:main}]
If $P_n=\varnothing$, the conclusion is immediate. Otherwise, by
\eqref{eq:reduction:circle} and \eqref{eq:scaling}, assume that $n$ is a
product of split prime powers. Among its finitely many positive divisors,
choose $d\mid n$ such that
\begin{equation}\label{eq:closure:maximum}
 L:=\frac{J_3(P_d)}{M_d^3}
   =\max_{k\mid n}\frac{J_3(P_k)}{M_k^3}.
\end{equation}
It suffices to bound $L$. We may assume $L>16$.
Since the first five entries determine the sixth,
$J_3(P_1)\le4^5=16M_1^3$, so $d>1$.
Let $s$ be the number of distinct prime factors of $d$, and write
$d=p^\nu m$, where $p$ is the largest prime factor and $p\nmid m$.
By \eqref{eq:product},
\begin{equation}\label{eq:closure:children}
 M_d=(\nu+1)M_m,\qquad M_{d/p}=\nu M_m,\qquad 2M_m\le M_d.
\end{equation}

To bound $L$, it suffices to bound $R_p(m)/M_m$.
Indeed, since $m,d/p\mid n$, \eqref{eq:closure:maximum} and
\eqref{eq:closure:children} give $J_3(P_d)=L(\nu+1)^3M_m^3$, $J_3(P_m)\le LM_m^3$ and
$J_3(P_{d/p})\le L\nu^3M_m^3$. Substitution into \eqref{eq:recurrence}
and division by $M_m^3$ yield
\[L\bigl((\nu+1)^3-\nu^3-2\bigr)
\le C_1\frac{R_p(m)}{M_m}L^{2/3}\nu^2.\]
Since $(\nu+1)^3-\nu^3-2=3\nu^2+3\nu-1\ge3\nu^2$ for $\nu\ge1$,
we obtain
\begin{equation} \label{eq:closure:bound}
 L\le\Bigl(\frac{C_1}{3}\Bigr)^3
       \Bigl(\frac{R_p(m)}{M_m}\Bigr)^3.
\end{equation}

By \eqref{eq:split:R} and \eqref{eq:local}, it then suffices to bound
$\rho_p(P_m)\log(2M_m)$, since
\begin{equation}\label{eq:closure:R}
 \Bigl(\frac{R_p(m)}{M_m}\Bigr)^3
 \le\frac43+16C_0\rho_p(P_m)\log(2M_m).
\end{equation}
By \eqref{eq:rho} and \eqref{eq:closure:children}, this reduces to
bounding $\log M_d/(s+1)$:
\begin{equation}\label{eq:closure:residue}
 \rho_p(P_m)\log(2M_m)\le\frac{\log M_d}{s+1}.
\end{equation}
Finally, \eqref{eq:ESS} gives
\begin{equation}\label{eq:closure:range}
 16<L\le15+\exp\bigl(C_*(s+1)-\log M_d\bigr)
 \quad\Longrightarrow\quad \log M_d<C_*(s+1).
\end{equation}
Combining \eqref{eq:closure:maximum} and
\eqref{eq:closure:bound}--\eqref{eq:closure:range}, we obtain
\[
 \frac{J_3(P_n)}{M_n^3}\le L
 \le\Bigl(\frac{C_1}{3}\Bigr)^3\Bigl(\frac43+16C_0C_*\Bigr).
\]
Together with the case $L\le16$, this proves \eqref{eq:main}.
\end{proof}

\subsection{Proof of Proposition~\ref{prop:recurrence}}\label{sec:split}
\begin{proof}
Let $\Lambda=\bar\varpi\Z[i]$. Unique factorization gives
\begin{equation}\label{eq:split:sets}
 P_{p^\nu m}=A\sqcup B,\qquad
 A=\bar\varpi P_{p^{\nu-1}m},\qquad B=\varpi^\nu P_m.
\end{equation}
Here $A\subset\Lambda$ and $B\cap\Lambda=\varnothing$: if
$\bar\varpi\mid\varpi^\nu u$ with $u\in P_m$, coprimality forces
$\bar\varpi\mid u$, contradicting $p\nmid\N(u)=m$.
By \eqref{eq:circle:symmetries}, $f_A,f_B$ are real, and
\eqref{eq:projection:identity} gives
\begin{equation}\label{eq:split:projections}
 E_\Lambda f_A=f_A,\qquad E_\Lambda f_B=0.
\end{equation}
For $u,v\in P_m$, $u$ is invertible modulo $\bar\varpi$ since $p\nmid m$.
The identity $u\bar u=v\bar v=m$ therefore gives
\[
 \begin{gathered}
 u(\bar u-\bar v)=-(u-v)\bar v,\\
 \bar\varpi\mid u-v\ \Longrightarrow\
 \bar\varpi\mid\bar u-\bar v\ \Longrightarrow\ \varpi\mid u-v.
 \end{gathered}
\]
By coprimality of $\varpi,\bar\varpi$,
\begin{equation}\label{eq:first:congruence}
 \bar\varpi\mid\varpi^\nu(u-v)
 \quad\Longleftrightarrow\quad\bar\varpi\mid u-v
 \quad\Longleftrightarrow\quad p\mid u-v.
\end{equation}
Since $f_B^2=|f_B|^2$, \eqref{eq:fourier:counts},
\eqref{eq:projection:identity}, and \eqref{eq:first:congruence} yield
\begin{equation}\label{eq:split:projected-square}
 E_\Lambda(f_B^2)(x)
 =\sum_{\substack{u,v\in P_m\\p\mid u-v}}
       e^{2\pi i\varpi^\nu(u-v)\cdot x}
 =\bigl(E_{p\Z^2}f_{P_m}^{\,2}\bigr)(A_{\varpi^\nu}^{T}x).
\end{equation}
Thus \eqref{eq:scaling} and \eqref{eq:split:R} give
\begin{equation}\label{eq:split:quadratic-norm}
 \|E_\Lambda(f_B^2)\|_3=\|E_{p\Z^2}f_{P_m}^{\,2}\|_3=R_p(m).
\end{equation}
Since every power of $f_A$ has Fourier support in $\Lambda$,
\eqref{eq:projection:module} and \eqref{eq:split:projections} imply
\[
 \int f_A^5f_B=0,\qquad
 \int f_A^4f_B^2=\int f_A^4E_\Lambda(f_B^2).
\]
For an absolute constant $C_1\ge1$,
\begin{equation}\label{eq:split:scalar}
 (a+b)^6\le a^6+6a^5b+2b^6+C_1a^4b^2
 \qquad(a,b\in\mathbb R).
\end{equation}
Integrating \eqref{eq:split:scalar} with $a=f_A$, $b=f_B$, and applying
H\"older, \eqref{eq:energy}, \eqref{eq:scaling}, and
\eqref{eq:split:quadratic-norm}, we obtain
\begin{equation}\label{eq:mixed:moments}
 \begin{aligned}
 J_3(P_{p^\nu m})
 &\le\|f_A\|_6^6+2\|f_B\|_6^6
       +C_1\|f_A\|_6^4\|E_\Lambda(f_B^2)\|_3\\
 &=J_3(P_{p^{\nu-1}m})+2J_3(P_m)
       +C_1J_3(P_{p^{\nu-1}m})^{2/3}R_p(m),
 \end{aligned}
\end{equation}
which is \eqref{eq:recurrence}.
\end{proof}

\subsection{Proof of Proposition~\ref{prop:local}}\label{sec:local}
\begin{lemma}\label{lem:local:cubic}
Let $p$ be an odd prime, $I$ a finite index set, and $(H_c)_{c\in I}$
nonnegative trigonometric polynomials on $\T^2$ with $\widehat H_c\ge0$.
Set $E_k=E_{p^k\Z^2}$ for $k\ge0$, and define
\begin{equation}\label{eq:local:increments:definition}
 \begin{gathered}
 H=\sum_{c\in I}H_c,\qquad
 R=\Bigl(\sum_{c\in I}\|H_c\|_3^{3/2}\Bigr)^{2/3},\\
 d_{c,k}=(E_k-E_{k+1})H_c,\qquad d_k=\sum_{c\in I}d_{c,k}.
 \end{gathered}
\end{equation}
Let $(\ell_c)_{c\in I}$ be pairwise distinct one-dimensional subspaces
of $\mathbb F_p^2$. Assume that, for every $c\in I$ and $k\ge0$,
\begin{equation}\label{eq:local:directions}
 \operatorname{supp}\widehat d_{c,k}
 \subset\{p^kv:v\in\Z^2,\ v\bmod p\in\ell_c\setminus\{0\}\}.
\end{equation}
Then
\begin{equation}\label{eq:positive:bound}
 \|H\|_3^3\le\tfrac43\Bigl(\int H\Bigr)^3+8R^3.
\end{equation}
\end{lemma}

\begin{proof}
All sums over $c$ are over $I$. By \eqref{eq:projection:identity},
\begin{equation}\label{eq:local:increment-multiplier}
 \widehat d_{c,k}(\xi)
 =\mathbf1_{p^k\Z^2\setminus p^{k+1}\Z^2}(\xi)\widehat H_c(\xi)\ge0.
\end{equation}
Finite Fourier support gives $E_kH_c=\int H_c$ and $d_{c,k}=0$ for
large $k$. Equations~\eqref{eq:projection:properties} and
\eqref{eq:projection:module} give $E_{k+1}d_{c,k}=0$ and
\[
 \int(E_{k+1}H_c)^2d_{c,k}
 =\int(E_{k+1}H_c)^2E_{k+1}d_{c,k}=0.
\]
Apply $|a-b|^3\le a^3-b^3-3b^2(a-b)$ for $a,b\ge0$ to
$a=E_kH_c$ and $b=E_{k+1}H_c$. Positivity in
\eqref{eq:projection:properties} and the preceding cancellation yield
\begin{equation}\label{eq:increments:step}
 \|d_{c,k}\|_3^3\le\|E_kH_c\|_3^3-\|E_{k+1}H_c\|_3^3.
\end{equation}
Since $E_0$ is the identity, summing \eqref{eq:increments:step} gives
\begin{equation}\label{eq:increments}
 \sum_{k\ge0}\|d_{c,k}\|_3^3
 \le\|H_c\|_3^3-\Bigl(\int H_c\Bigr)^3\le\|H_c\|_3^3.
\end{equation}

For fixed $k$, let
\[
 \Lambda_{c,k}=\{p^kv:v\in\Z^2,\ v\bmod p\in\ell_c\},\qquad
 \Pi_{c,k}=E_{\Lambda_{c,k}}-E_{k+1}.
\]
Since $p^{k+1}\Z^2\subset\Lambda_{c,k}$, the multiplier of $\Pi_{c,k}$
is $\mathbf1_{\Lambda_{c,k}\setminus p^{k+1}\Z^2}$, so
\eqref{eq:local:directions} gives $\Pi_{c,k}d_{c,k}=d_{c,k}$.
Since distinct lines meet only at zero, these projections have disjoint
frequency supports. Their $L^\infty$ norms are at most $2$ by
\eqref{eq:projection:properties}. Thus \eqref{eq:one:scale}, Minkowski's
inequality in $\ell^2$, and \eqref{eq:increments} give
\begin{equation}\label{eq:cubic:increments}
 \begin{aligned}
 \sum_k\|d_k\|_3^3
 &\le2\sum_k\Bigl(\sum_c\|d_{c,k}\|_3^{3/2}\Bigr)^2\\
 &\le2\bigl[\sum_c\Bigl(\sum_k\|d_{c,k}\|_3^3\Bigr)^{1/2}\bigr]^2
 \le2R^3.
 \end{aligned}
\end{equation}

Expand $\int H^3=\sum_{u+v+w=0}\widehat H(u)\widehat H(v)\widehat H(w)$.
Apart from $(0,0,0)$, each triple has a unique $k\ge0$ such that
$u,v,w\in p^k\Z^2$, but not all lie in $p^{k+1}\Z^2$.
At least two lie outside $p^{k+1}\Z^2$.
If exactly two do, relabel so that $w\in p^{k+1}\Z^2$. Then
$p^{-k}u\equiv-p^{-k}v\not\equiv0\pmod p$.
By \eqref{eq:local:directions} and \eqref{eq:local:increment-multiplier},
$\widehat H_c(u)\widehat H_{c'}(v)>0$ forces $\ell_c=\ell_{c'}$, hence
$c=c'$. Therefore
\[
 \widehat H(u)\widehat H(v)
 =\sum_c\widehat H_c(u)\widehat H_c(v).
\]
For each position of $w$, the triples for different $k$ are disjoint.
Summing over $k$ and then dropping the divisibility restrictions,
using $\widehat H_c\ge0$, bounds this contribution by
\[
 \sum_c\sum_{u+v+w=0}\widehat H_c(u)\widehat H_c(v)\widehat H(w)
 =\sum_c\int HH_c^2.
\]
If all three frequencies lie outside $p^{k+1}\Z^2$, their contribution
is exactly $\int d_k^3$ by \eqref{eq:local:increment-multiplier}.
Including $(0,0,0)$ and the three possible positions of $w$ gives
\begin{equation}\label{eq:local:fourier:cubic}
 \int H^3\le\Bigl(\int H\Bigr)^3
       +3\sum_c\int HH_c^2+\sum_k\int d_k^3.
\end{equation}
Since $\sum_c\|H_c\|_3^2\le R^2$, H\"older and
\eqref{eq:cubic:increments}--\eqref{eq:local:fourier:cubic}, followed by
$3ab^2\le a^3/4+4b^3$ for $a,b\ge0$, yield
\begin{equation}\label{eq:local:cubic:inequality}
 \begin{aligned}
 \|H\|_3^3
 &\le\Bigl(\int H\Bigr)^3+3\|H\|_3R^2+2R^3\\
 &\le\Bigl(\int H\Bigr)^3+\tfrac14\|H\|_3^3+6R^3.
 \end{aligned}
\end{equation}
Rearranging proves \eqref{eq:positive:bound}.
\end{proof}

\begin{proof}[Proof of Proposition~\ref{prop:local}]
For each nonempty $S_c$ in \eqref{eq:local:classes},
$c\cdot c=m\ne0$ in $\mathbb F_p$, so $c\ne-c$.
Let $I$ contain one representative of each such pair $\{c,-c\}$.
All sums over $c$ below are over $I$.
Take $H_c=2|f_{S_c}|^2$ in Lemma~\ref{lem:local:cubic}.
Since $S_{-c}=-S_c$ and $f_{P_m}$ is real, \eqref{eq:fourier:counts}
and \eqref{eq:projection:identity} give
\begin{equation}\label{eq:local:polynomials}
 \begin{aligned}
 H&=\sum_cH_c=E_{p\Z^2}f_{P_m}^{\,2},\\
 \widehat H_c(v)&=2\#\{(a,b)\in S_c^2:a-b=v\}\ge0.
 \end{aligned}
\end{equation}
Parseval and the partition of $P_m$ into the pairs $S_c\sqcup S_{-c}$ give
\begin{equation}\label{eq:local:mass:single}
 \int H_c=\widehat H_c(0)=2|S_c|,
\end{equation}
\begin{equation}\label{eq:masses}
 \sum_c|S_c|=M_m/2,\qquad \max_c|S_c|\le\rho_p(P_m)M_m,
\end{equation}
and hence
\begin{equation}\label{eq:local:mass:total}
 \int H=2\sum_c|S_c|=M_m.
\end{equation}

Let $\ell_c=\{v\in\mathbb F_p^2:c\cdot v=0\}$, which is one-dimensional
since $c\ne0$.
If $\widehat d_{c,k}(\xi)\ne0$, equations
\eqref{eq:local:increment-multiplier} and \eqref{eq:local:polynomials}
give $a,b\in S_c$ with $a-b=\xi=p^kv$ and $v\not\equiv0\pmod p$.
Their equal norms and $a\equiv b\equiv c\pmod p$ imply
\begin{equation}\label{eq:local:circle:identity}
 0=p^{-k}(|a|^2-|b|^2)=(a+b)\cdot v\equiv2c\cdot v\pmod p.
\end{equation}
Since $p$ is odd, $v\bmod p\in\ell_c\setminus\{0\}$, proving
\eqref{eq:local:directions}.
If $\ell_{c'}=\ell_c$, then $c'=tc$ for some $t\in\mathbb F_p^\times$.
The identities $c\cdot c=c'\cdot c'=m\ne0$ give $t^2=1$, hence $c'=\pm c$.
Our choice of $I$ forces $c'=c$, so the lines are pairwise distinct
and Lemma~\ref{lem:local:cubic} applies.

By \eqref{eq:energy}, $\|H_c\|_3=2J_3(S_c)^{1/3}$.
Also, $M_m\ge4$ by \eqref{eq:circle:symmetries}, so
$\log(2+|S_c|)\le\log(2M_m)$.
Thus \eqref{eq:local:increments:definition} and \eqref{eq:log-energy} give
\begin{equation}\label{eq:local:mass:energy}
 \begin{aligned}
 R^3
 &=8\Bigl(\sum_cJ_3(S_c)^{1/2}\Bigr)^2\\
 &\le8C_0\log(2M_m)\Bigl(\sum_c|S_c|^{3/2}\Bigr)^2.
 \end{aligned}
\end{equation}
By \eqref{eq:masses},
\[
 \Bigl(\sum_c|S_c|^{3/2}\Bigr)^2
 \le(\max_c|S_c|)\Bigl(\sum_c|S_c|\Bigr)^2
 \le\tfrac14\rho_p(P_m)M_m^3.
\]
Combining this with \eqref{eq:local:mass:energy} yields
\begin{equation}\label{eq:local:mass:bound}
 R^3\le2C_0\rho_p(P_m)M_m^3\log(2M_m).
\end{equation}
Finally, \eqref{eq:split:R}, \eqref{eq:local:polynomials},
\eqref{eq:local:mass:total}, and \eqref{eq:positive:bound} give
\[
 R_p(m)^3=\|H\|_3^3
 \le\tfrac43M_m^3+8R^3
 \le M_m^3\Bigl(\tfrac43+16C_0\rho_p(P_m)\log(2M_m)\Bigr),
\]
where the last inequality uses \eqref{eq:local:mass:bound}.
This proves \eqref{eq:local}.
\end{proof}

\subsection{Proof of Proposition~\ref{prop:residues}}\label{sec:concentration}
\begin{proof}
Fix a nonempty $S_c$ in \eqref{eq:local:classes}. For $j<r$, choose
$\varpi_j$ as in \eqref{eq:product}, with positive real and imaginary
parts, and set
\begin{equation}\label{eq:packing:representatives}
 \mathcal U=\{1,1+i\}\cup\{\varpi_j,\bar\varpi_j:1\le j<r\}.
\end{equation}
For $\mathcal D=\{(z,u)\in S_c\times\mathcal U:\bar u\mid uz\}$, define
\begin{equation}\label{eq:packing:images}
 \Phi:\mathcal D\longrightarrow P_m,\qquad
 \Phi(z,u)=\frac{uz}{\bar u}.
\end{equation}
This is well-defined since $\bar u\mid uz$ and
$\N(uz/\bar u)=\N(z)=m$. We prove that $\Phi$ is injective and that
at least $r+1$ choices of $u$ are admissible for each $z\in S_c$. Then
\begin{equation}\label{eq:packing:count}
 M_m\ge|\mathcal D|\ge(r+1)|S_c|,
\end{equation}
which gives \eqref{eq:rho} by \eqref{eq:local:classes}.

For injectivity, note that each $u\in\mathcal U$ has coprime integer
coordinates and positive real part. Its norm is $1$, $2$, or some
$p_j<p$, so $\N(u)<p$. Suppose $\Phi(z,u)=\Phi(w,v)$, equivalently
$u\bar v\,z=v\bar u\,w$. Since $z\equiv w\pmod p$ and
$z^{-1}\equiv m^{-1}\bar z\pmod p$ by Section~\ref{subsec:gaussian},
\begin{equation}\label{eq:packing:determinant}
 \begin{gathered}
 u\bar v-v\bar u=2i\Im(u\bar v)\equiv0\pmod p,\\
 |\Im(u\bar v)|\le\sqrt{\N(u)\N(v)}<p.
 \end{gathered}
\end{equation}
Since $p$ is odd, \eqref{eq:packing:determinant} forces
$\Im(u\bar v)=0$, hence $u/v\in\mathbb Q$. Coprime coordinates give
$u=\pm v$, and positive real parts give $u=v$. Thus $z=w$.

For each $z\in S_c$, the choices $u=1,1+i$ are admissible, since
$(1+i)/(1-i)=i$. For each $j<r$, \eqref{eq:product} gives
$\varpi_j\mid z$ or $\bar\varpi_j\mid z$. Taking $u=\bar\varpi_j$
or $u=\varpi_j$, respectively, gives $\bar u\mid uz$.
These choices have distinct norms $1,2,p_1,\ldots,p_{r-1}$, so there
are at least $r+1$ of them. This completes the proof of
\eqref{eq:packing:count}.
\end{proof}

\subsection{Proof of Proposition~\ref{prop:unit}}\label{sec:unit}
The following lemma is from Bourgain's six-term estimate
\cite[Lemma~5]{BourgainRandom}, which follows from the
Evertse--Schlickewei--Schmidt theorem \cite[Theorem~1.1]{ESS}.
\begin{lemma}\label{lem:unit:nondegenerate}
Let $n\ge1$ have $r$ distinct prime factors, all split. Let $N_6$ count the
ordered sextuples in $P_n^6$ with zero sum and no nonempty proper vanishing subsum.
For an absolute constant $C_*>0$,
\begin{equation}\label{eq:unit:nondegenerate}
 N_6\le M_ne^{C_*(r+1)}.
\end{equation}
\end{lemma}
\begin{proof}[Proof of Proposition~\ref{prop:unit}]
By \eqref{eq:energy:def} and \eqref{eq:circle:symmetries}, negating the
last three entries shows that $J_3(P_n)$ counts zero-sum sextuples.
We show that every sextuple not counted by $N_6$ splits into three
opposite pairs.

Since $n$ is odd, each $z=a+ib\in P_n$ satisfies
$a+b\equiv a^2+b^2=n\equiv1\pmod2$. Using $i\equiv1\pmod{1+i}$ gives
\begin{equation}\label{eq:unit:parity}
 z=a+ib\equiv a+b\equiv1\pmod{1+i}.
\end{equation}
For integers $k$, $1+i\mid k$ if and only if $2\mid k$. Thus
\eqref{eq:unit:parity} makes every vanishing subsum even in length.
A nonempty proper vanishing subsum has size $2$ or $4$, so it or its
complement is an opposite pair.

The remaining four points $u,v,w,z\in P_n$ satisfy $u+v+w+z=0$.
Using $\bar u=n/u$ and the analogous identities from
\eqref{eq:circle:def}, we obtain
\begin{equation}\label{eq:unit:four-points}
 \begin{aligned}
 (u+v)(u+w)(v+w)
 &=-(uvw+uvz+uwz+vwz)\\
 &=-\frac{uvwz}{n}(\bar u+\bar v+\bar w+\bar z)=0.
 \end{aligned}
\end{equation}
One factor vanishes, so these four points form two opposite pairs.
There are $6!/((2!)^3\,3!)=15$ pairings of six positions, each allowing
$M_n^3$ choices. Hence \eqref{eq:unit:nondegenerate} gives
\begin{equation}\label{eq:unit:partition}
 J_3(P_n)\le15M_n^3+N_6
 \le15M_n^3+M_ne^{C_*(r+1)}.
\end{equation}
Since $M_n\ge4$ by \eqref{eq:product}, division by $M_n^3$ yields \eqref{eq:ESS}.
\end{proof}

\section{Higher additive energies}\label{sec:higher-energy}
 We improve \eqref{eq:higher} by a fixed power
for every $k\ge4$. 
The proof uses Theorem~\ref{thm:main} and a weighted version of the
residue-class argument in Proposition~\ref{prop:residues}. For simplicity, we do not attempt to optimize the power-saving exponent.  Reaching the conjectured bound \eqref{eq:higher:conjecture}
appears to require new methods. 

\begin{theorem}\label{thm:higher:power}
There are absolute constants $\delta>0$ and $C$ such that, for every
$n\ge1$ and every integer $k\ge4$,
\begin{equation}\label{eq:higher:power}
 J_k(P_n)\le C |P_n|^{2k-3-\delta}.
\end{equation}
\end{theorem}

As before, we denote $M_n=|P_n|$. By \eqref{eq:reduction:circle} and \eqref{eq:scaling}, it suffices to
consider products of split prime powers. Use \eqref{eq:product} and write
\begin{equation}\label{eq:higher:parameters}
 n=\prod_{j=1}^r p_j^{\nu_j},\qquad
 p_j=\varpi_j\bar\varpi_j,\qquad
 M_n=4\prod_{j=1}^r(\nu_j+1),
\end{equation}
where the $p_j$ are distinct split primes and $\nu_j\ge1$.
For a Gaussian prime $\alpha$ and $0\ne z\in\Z[i]$, let
$v_\alpha(z)=\max\{b\ge0:\alpha^b\mid z\}$, with $v_\alpha(0)=+\infty$.
For $s\in\Z[i]$ and $z\in P_n$, define
\begin{equation}\label{eq:higher:weights}
 \begin{gathered}
 \mathcal H_s(z)=\{j\in\{1,\ldots,r\}:
    v_{\varpi_j}(s-z)>v_{\varpi_j}(z)\},\\
 h_s(z)=\frac12\sum_{j\in\mathcal H_s(z)}\log(\nu_j+1).
 \end{gathered}
\end{equation}

\begin{lemma}\label{lem:higher:weighted}
For every $n$ in \eqref{eq:higher:parameters} and every $s\in\Z[i]$,
\begin{equation}\label{eq:higher:weighted}
 \sum_{z\in P_n}e^{h_s(z)}\le M_n(1+\log M_n).
\end{equation}
\end{lemma}

The following two propositions isolate the counting arguments. We first
deduce Theorem~\ref{thm:higher:power} and Lemma~\ref{lem:higher:weighted},
then prove the propositions.

For $u\in\Z[i]$ and $\ell\in\{3,4\}$, let
\begin{equation}\label{eq:higher:representation-counts}
 r_\ell(u)=\#\bigl\{(y_1,\ldots,y_\ell)\in P_n^\ell:
                           \sum_{b=1}^\ell y_b=u\bigr\}.
\end{equation}
All sums in $u$ below are over $\Z[i]$.

\begin{proposition}\label{prop:higher:eight-point}
For every $n$ in \eqref{eq:higher:parameters},
\begin{equation}\label{eq:higher:energy-sums}
 \sum_u r_4(-u)^2=J_4(P_n),\qquad
 \sum_u\sum_{z\in P_n}r_3(u-z)^2=M_nJ_3(P_n).
\end{equation}
There is an absolute $\delta\in(0,1)$ such that, for every such $n$,
\begin{equation}\label{eq:higher:weighted-eight}
 J_4(P_n)^{2+\delta}\le4^{6\delta}M_n^{2\delta}
 \bigl(\sum_u\sum_{z\in P_n}
             r_4(-u)e^{h_u(z)/2}r_3(u-z)\bigr)^2.
\end{equation}
\end{proposition}

\begin{proposition}\label{prop:higher:rank-count}
For $n$ in \eqref{eq:higher:parameters} and $I\subseteq\{1,\ldots,r\}$, let
\begin{equation}\label{eq:higher:rank-definition}
 Q_I=\prod_{j\in I}p_j,\qquad Q_\varnothing=1,\qquad
 R(I)=\#\{B\subseteq\{1,\ldots,r\}:Q_B<Q_I\}.
\end{equation}
For every $s\in\Z[i]$ and every nonempty $I\subseteq\{1,\ldots,r\}$,
\begin{equation}\label{eq:higher:rank}
 \#\{z\in P_n:I\subseteq\mathcal H_s(z)\}
 \le\frac{M_n}{R(I)}\prod_{j\in I}\frac2{\nu_j+1}.
\end{equation}
\end{proposition}

\begin{proof}[Proof of Theorem~\ref{thm:higher:power}]
Empty circles are immediate. By \eqref{eq:reduction:circle} and
\eqref{eq:scaling}, assume that $n$ has the form
\eqref{eq:higher:parameters}. Equations \eqref{eq:fourier:sum} and
\eqref{eq:energy} give
\begin{equation}\label{eq:higher:interpolation}
 J_k(P_n)=\int_{\T^2}|f_{P_n}|^{2k-8}|f_{P_n}|^8
 \le M_n^{2k-8}J_4(P_n)\qquad(k\ge4).
\end{equation}
Let $\delta$ be as in Proposition~\ref{prop:higher:eight-point}.
By \eqref{eq:higher:weighted-eight}, Cauchy--Schwarz,
\eqref{eq:higher:energy-sums}, \eqref{eq:higher:weighted}, and
\eqref{eq:main},
\begin{equation}\label{eq:higher:energy-closure}
 \begin{aligned}
 J_4(P_n)^{2+\delta}
 &\le4^{6\delta}M_n^{2\delta}
    \bigl(\sum_u r_4(-u)^2\sum_{z\in P_n}e^{h_u(z)}\bigr)
    \bigl(\sum_u\sum_{z\in P_n}r_3(u-z)^2\bigr)\\
 &\le4^{6\delta}J_4(P_n)M_n^{2+2\delta}(1+\log M_n)J_3(P_n)\\
 &\le CJ_4(P_n)M_n^{5+2\delta}(1+\log M_n).
 \end{aligned}
\end{equation}
Since $J_4(P_n)\ge M_n^4>0$ and $1+\log t\le C_\delta t^\delta$ for
$t\ge4$,
\[
 J_4(P_n)^{1+\delta}\le CM_n^{5+3\delta}
 \le CM_n^{(5-\delta)(1+\delta)},\qquad
 (5-\delta)(1+\delta)-(5+3\delta)=\delta(1-\delta)\ge0.
\]
Taking $(1+\delta)$th roots and using \eqref{eq:higher:interpolation}
proves \eqref{eq:higher:power}, with the same constant for all $k\ge4$.
\end{proof}

\begin{proof}[Proof of Lemma~\ref{lem:higher:weighted}]
Fix $s\in\Z[i]$. Since $\sqrt t\le1+t/2$ for $t\ge1$,
\eqref{eq:higher:weights} gives
\begin{equation}\label{eq:higher:weight-expansion}
 e^{h_s(z)}=\prod_{j\in\mathcal H_s(z)}\sqrt{\nu_j+1}
 \le\prod_{j\in\mathcal H_s(z)}\Bigl(1+\frac{\nu_j+1}{2}\Bigr)
 =\sum_{I\subseteq\mathcal H_s(z)}\prod_{j\in I}\frac{\nu_j+1}{2}.
\end{equation}
By unique factorization, $Q_B=Q_I$ implies $B=I$, so
$\{R(I):\varnothing\ne I\subseteq\{1,\ldots,r\}\}
=\{1,\ldots,2^r-1\}$. Summing \eqref{eq:higher:weight-expansion}
and using \eqref{eq:higher:rank},
\begin{equation}\label{eq:higher:weighted-sum}
 \begin{aligned}
 \sum_{z\in P_n}e^{h_s(z)}
 &\le M_n+\sum_{\varnothing\ne I\subseteq\{1,\ldots,r\}}
   \Bigl(\prod_{j\in I}\frac{\nu_j+1}{2}\Bigr)
   \#\{z\in P_n:I\subseteq\mathcal H_s(z)\}\\
 &\le M_n\Bigl(1+\sum_{\varnothing\ne I\subseteq\{1,\ldots,r\}}
                 \frac1{R(I)}\Bigr)
 =M_n\Bigl(1+\sum_{b=1}^{2^r-1}\frac1b\Bigr)
 \le M_n(1+\log M_n).
 \end{aligned}
\end{equation}
The last step uses $1+\sum_{b=1}^{2^r-1}b^{-1}\le2+\log(2^r)
\le1+\log M_n$ by \eqref{eq:higher:parameters}, also for $r=0$.
\end{proof}

\subsection{Proof of Proposition~\ref{prop:higher:eight-point}}
\label{subsec:higher:eight-point-proof}
\begin{proof}
By \eqref{eq:energy:def} and \eqref{eq:higher:representation-counts},
\[
 \sum_u r_\ell(u)^2
 =\#\bigl\{(\mathbf x,\mathbf y)\in P_n^\ell\times P_n^\ell:
              \sum_{b=1}^\ell x_b=\sum_{b=1}^\ell y_b\bigr\}
 =J_\ell(P_n)\qquad(\ell=3,4).
\]
The changes of variables $u\mapsto-u$ and, for fixed $z$, $u\mapsto u-z$
give \eqref{eq:higher:energy-sums}.
For $\mathbf z=(z_1,\ldots,z_8)$, let
$\overline{\mathbf z}=(\bar z_1,\ldots,\bar z_8)$ and
\begin{equation}\label{eq:higher:tuples}
 \Omega_n=\bigl\{\mathbf z\in P_n^8:\sum_{b=1}^8z_b=0\bigr\},
 \qquad |\Omega_n|=J_4(P_n)\ge M_n^4>0,
 \qquad\overline{\Omega_n}=\Omega_n.
\end{equation}
Here \eqref{eq:energy:def} and \eqref{eq:circle:symmetries} apply, and
$(z_1,\ldots,z_4,-z_1,\ldots,-z_4)\in\Omega_n$ for all
$z_1,\ldots,z_4\in P_n$. Let
\[
 \mathcal C=\{(A,b):A\subseteq\{1,\ldots,8\},\ |A|=4,\ b\in A\},
 \qquad |\mathcal C|=4\binom84,
\]
and fix $0<8\delta\le(4|\mathcal C|)^{-1}$, so $\delta<1$.
For $\mathbf z\in P_n^8$ and $A\subseteq\{1,\ldots,8\}$, let
\[
 s_A(\mathbf z)=\sum_{a\in A}z_a,\qquad
 s_A(\overline{\mathbf z})=\overline{s_A(\mathbf z)}.
\]
For $(A,b)\in\mathcal C$ and $\mathbf z\in\Omega_n$,
\[
 \sum_{a\in A\setminus\{b\}}z_a=s_A(\mathbf z)-z_b,
 \qquad \sum_{a\notin A}z_a=-s_A(\mathbf z).
\]
Thus, for every $(A,b)\in\mathcal C$, $z\in P_n$, and $u\in\Z[i]$,
\begin{equation}\label{eq:higher:joint-count}
 \#\{\mathbf z\in\Omega_n:z_b=z,\ s_A(\mathbf z)=u\}
 =r_3(u-z)r_4(-u).
\end{equation}

For $\mathbf z\in P_n^8$, let
\[
 \begin{gathered}
 k_{j,b}=v_{\varpi_j}(z_b)\in\{0,\ldots,\nu_j\}
 \qquad(1\le j\le r,\ 1\le b\le8),\\
 \mathcal J(\mathbf z)=\{j\in\{1,\ldots,r\}:
             \exists b\in\{2,\ldots,8\},\ k_{j,b}\ne k_{j,1}\}.
 \end{gathered}
\]
Define
\begin{equation}\label{eq:higher:variation-definition}
 \mathcal V(\mathbf z)=\sum_{j\in\mathcal J(\mathbf z)}\log(\nu_j+1),
 \qquad e^{-8\mathcal V(\mathbf z)}
       =\prod_{j\in\mathcal J(\mathbf z)}(\nu_j+1)^{-8}.
\end{equation}
The bijective parameterization in \eqref{eq:product} is
\[
 z_b=\varepsilon_b\prod_{j=1}^r
 \varpi_j^{k_{j,b}}\bar\varpi_j^{\nu_j-k_{j,b}},\qquad
 \varepsilon_b\in\{\pm1,\pm i\}\quad(1\le b\le8).
\]
By \eqref{eq:higher:variation-definition}, summing first over the units
and then over the exponent rows gives
\begin{equation}\label{eq:higher:variation-product}
 \begin{aligned}
 &\sum_{\mathbf z\in\Omega_n}e^{-8\mathcal V(\mathbf z)}
 \le\sum_{\mathbf z\in P_n^8}e^{-8\mathcal V(\mathbf z)}\\
 &=4^8\!\sum_{\substack{0\le k_{j,b}\le\nu_j\\1\le j\le r,\ 1\le b\le8}}
   \prod_{j=1}^r\bigl[(\nu_j+1)^{-8}
    +\bigl(1-(\nu_j+1)^{-8}\bigr)
                   \mathbf1_{\{k_{j,1}=\cdots=k_{j,8}\}}\bigr]\\
 &=4^8\prod_{j=1}^r
   \sum_{k_{j,1},\ldots,k_{j,8}=0}^{\nu_j}
   \bigl[(\nu_j+1)^{-8}
    +\bigl(1-(\nu_j+1)^{-8}\bigr)
                   \mathbf1_{\{k_{j,1}=\cdots=k_{j,8}\}}\bigr]\\
 &=4^8\prod_{j=1}^r
       \bigl[1+\bigl(1-(\nu_j+1)^{-8}\bigr)(\nu_j+1)\bigr]
 \le4^8\prod_{j=1}^r(\nu_j+1)^2=4^6M_n^2.
 \end{aligned}
\end{equation}

Fix $\mathbf z\in P_n^8$ and $j\in\mathcal J(\mathbf z)$.
Since $(k_{j,b})_{b=1}^8$ is nonconstant, \eqref{eq:product} gives
\begin{equation}\label{eq:higher:extreme-count}
 \begin{gathered}
 \#\{b:k_{j,b}=\min_c k_{j,c}\}
 +\#\{b:k_{j,b}=\max_c k_{j,c}\}\le8,\\
 v_{\varpi_j}(\bar z_b)=\nu_j-k_{j,b},\qquad
 \min_b v_{\varpi_j}(\bar z_b)=\nu_j-\max_b k_{j,b}.
 \end{gathered}
\end{equation}
Here $a,b,c\in\{1,\ldots,8\}$.
For this $j$, \eqref{eq:higher:extreme-count} gives
$\mathbf y\in\{\mathbf z,\overline{\mathbf z}\}$ and $b$ such that
\[
 v_{\varpi_j}(y_b)=\min_c v_{\varpi_j}(y_c),\qquad
 \#\{a:v_{\varpi_j}(y_a)>v_{\varpi_j}(y_b)\}\ge4.
\]
Choose $(A,b)\in\mathcal C$ with
$A\setminus\{b\}\subseteq\{a:v_{\varpi_j}(y_a)>v_{\varpi_j}(y_b)\}$.

Using $v_{\varpi_j}(u+v)\ge\min\{v_{\varpi_j}(u),v_{\varpi_j}(v)\}$,
we obtain
\begin{equation}\label{eq:higher:marked-minimum}
 v_{\varpi_j}\bigl(s_A(\mathbf y)-y_b\bigr)
 =v_{\varpi_j}\Bigl(\sum_{a\in A\setminus\{b\}}y_a\Bigr)
 \ge\min_{a\in A\setminus\{b\}}v_{\varpi_j}(y_a)
 >v_{\varpi_j}(y_b).
\end{equation}
By \eqref{eq:higher:weights}, $j\in\mathcal H_{s_A(\mathbf y)}(y_b)$.
Thus, for every $1\le j\le r$,
\begin{equation}\label{eq:higher:successful-markings}
 \sum_{(A,b)\in\mathcal C}
 \bigl[\mathbf1_{\{j\in\mathcal H_{s_A(\mathbf z)}(z_b)\}}
       +\mathbf1_{\{j\in\mathcal H_{s_A(\overline{\mathbf z})}(\bar z_b)\}}\bigr]
 \ge\mathbf1_{\{j\in\mathcal J(\mathbf z)\}}.
\end{equation}
Multiply \eqref{eq:higher:successful-markings} by
$\frac12\log(\nu_j+1)$ and sum over $j$. By
\eqref{eq:higher:weights} and \eqref{eq:higher:variation-definition},
\begin{equation}\label{eq:higher:finite-average}
 \frac1{2|\mathcal C|}\sum_{(A,b)\in\mathcal C}
 \bigl[h_{s_A(\mathbf z)}(z_b)+h_{s_A(\overline{\mathbf z})}(\bar z_b)\bigr]
 \ge\frac1{4|\mathcal C|}\sum_{j\in\mathcal J(\mathbf z)}\log(\nu_j+1)
 =\frac{\mathcal V(\mathbf z)}{4|\mathcal C|}\ge8\delta\mathcal V(\mathbf z).
\end{equation}
Convexity of the exponential, \eqref{eq:higher:finite-average},
$\overline{\Omega_n}=\Omega_n$ in \eqref{eq:higher:tuples}, and
\eqref{eq:higher:joint-count} give
\begin{equation}\label{eq:higher:positive-moment}
 \begin{aligned}
 \sum_{\mathbf z\in\Omega_n}e^{4\delta\mathcal V(\mathbf z)}
 &\le\frac1{2|\mathcal C|}\sum_{(A,b)\in\mathcal C}
   \sum_{\mathbf z\in\Omega_n}
   \bigl[e^{h_{s_A(\mathbf z)}(z_b)/2}
        +e^{h_{s_A(\overline{\mathbf z})}(\bar z_b)/2}\bigr]\\
 &=\frac1{|\mathcal C|}\sum_{(A,b)\in\mathcal C}
       \sum_u\sum_{z\in P_n}r_4(-u)e^{h_u(z)/2}r_3(u-z)\\
 &=\sum_u\sum_{z\in P_n}r_4(-u)e^{h_u(z)/2}r_3(u-z).
 \end{aligned}
\end{equation}
H\"older with exponents $(2+\delta)/2$ and $(2+\delta)/\delta$ yields
\[
 \begin{aligned}
 J_4(P_n)
 &=\sum_{\mathbf z\in\Omega_n}
   \bigl(e^{4\delta\mathcal V(\mathbf z)}\bigr)^{2/(2+\delta)}
   \bigl(e^{-8\mathcal V(\mathbf z)}\bigr)^{\delta/(2+\delta)}\\
 &\le\bigl(\sum_{\mathbf z\in\Omega_n}
                 e^{4\delta\mathcal V(\mathbf z)}\bigr)^{2/(2+\delta)}
       \bigl(\sum_{\mathbf z\in\Omega_n}
                 e^{-8\mathcal V(\mathbf z)}\bigr)^{\delta/(2+\delta)}.
 \end{aligned}
\]
Raising to $2+\delta$ and applying \eqref{eq:higher:positive-moment}
and \eqref{eq:higher:variation-product} proves
\eqref{eq:higher:weighted-eight}.
\end{proof}

\subsection{Proof of Proposition~\ref{prop:higher:rank-count}}
\label{subsec:higher:rank-proof}
\begin{proof}
Fix $s\in\Z[i]$ and a nonempty $I\subseteq\{1,\ldots,r\}$.
Let $\mathcal A=\{z\in P_n:I\subseteq\mathcal H_s(z)\}$ and assume
$\mathcal A\ne\varnothing$. By \eqref{eq:higher:weights} and \eqref{eq:product},
\begin{equation}\label{eq:higher:fixed-valuations}
 \begin{gathered}
 v_{\varpi_j}(s)
 =\min\{v_{\varpi_j}(z),v_{\varpi_j}(s-z)\}
 =v_{\varpi_j}(z)\le\nu_j,\\
 \varpi_j^{v_{\varpi_j}(s)+1}\mid s-z
 \qquad(z\in\mathcal A,\ j\in I).
 \end{gathered}
\end{equation}
For $z\in\mathcal A$ and $I'\subseteq\{1,\ldots,r\}$ with $Q_{I'}<Q_I$,
let
\begin{equation}\label{eq:higher:rotation-factor}
 \bar g=
 \prod_{\substack{j\in I'\\\varpi_j\mid z}}\varpi_j
 \prod_{\substack{j\in I'\\\varpi_j\nmid z}}\bar\varpi_j.
\end{equation}
Since $v_{\varpi_j}(z)+v_{\bar\varpi_j}(z)=\nu_j\ge1$, we have
\begin{equation}\label{eq:higher:rotation}
 \bar g\mid z,\qquad \N(g)=Q_{I'}<Q_I,\qquad
 (z,I')\longmapsto\frac{gz}{\bar g}\in P_n.
\end{equation}
For $j\in I$, \eqref{eq:higher:fixed-valuations} and
\eqref{eq:higher:rotation-factor} give
\begin{equation}\label{eq:higher:image-valuations}
 v_{\varpi_j}\!\Bigl(\frac{gz}{\bar g}\Bigr)
 =v_{\varpi_j}(s)
  +\mathbf1_{\{j\in I'\}}\bigl(1-2\mathbf1_{\{\varpi_j\mid s\}}\bigr).
\end{equation}
There are at most two choices for each of these exponents. By
\eqref{eq:product} and 
\eqref{eq:higher:parameters}, the number of images is at most
\begin{equation}\label{eq:higher:image-count}
 4\cdot2^{|I|}\prod_{j\notin I}(\nu_j+1)
 =4\prod_{j=1}^r(\nu_j+1)\prod_{j\in I}\frac2{\nu_j+1}
 =M_n\prod_{j\in I}\frac2{\nu_j+1}.
\end{equation}
By \eqref{eq:higher:rank-definition}, the domain of the map in
\eqref{eq:higher:rotation} has cardinality $R(I)|\mathcal A|$.
Thus \eqref{eq:higher:image-count} reduces \eqref{eq:higher:rank}
to proving that the map in \eqref{eq:higher:rotation} is injective.
Suppose $(z,I')$ and $(z',I'')$ have the same image, with factors $g,g'$.
Then
\begin{equation}\label{eq:higher:collision}
 \begin{gathered}
 g\bar g'z=g'\bar gz',\\
 (g\bar g'-g'\bar g)s
 =g\bar g'(s-z)-g'\bar g(s-z').
 \end{gathered}
\end{equation}
For $j\in I$, \eqref{eq:higher:fixed-valuations} implies
\[
 \varpi_j^{v_{\varpi_j}(s)+1}\mid(g\bar g'-g'\bar g)s
 \quad\Longrightarrow\quad
 \varpi_j\mid g\bar g'-g'\bar g=2i\Im(g\bar g').
\]
Taking norms and using $p_j>2$ gives
$p_j\mid4\Im(g\bar g')^2$, hence $p_j\mid\Im(g\bar g')$.
Since these primes are distinct, \eqref{eq:higher:rotation} gives
\begin{equation}\label{eq:higher:determinant}
 Q_I\mid\Im(g\bar g'),\qquad
 |\Im(g\bar g')|\le\sqrt{Q_{I'}Q_{I''}}<Q_I.
\end{equation}
Thus $g\bar g'=g'\bar g\in\Z\setminus\{0\}$, so
\eqref{eq:higher:collision} gives $z'=z$. Moreover,
\[
 Q_{I'}Q_{I''}=\N(g\bar g')=(g\bar g')^2
 \quad\Longrightarrow\quad I'=I'',
\]
since $Q_{I'}$ and $Q_{I''}$ are squarefree.
\end{proof}

\section{Applications}\label{sec:applications}
We derive consequences for endpoint eigenfunction bounds, triple sumsets,
mean-square convergence of observables, and nodal length fluctuations.

\subsection{Endpoint \texorpdfstring{$L^6$}{L6} estimates}\label{subsec:lp:l6}
We recover the  bound in \eqref{eq:intro:LY-weighted}.

\begin{corollary}\label{prop:lp:strong}
For $F$ in \eqref{eq:app:weighted-sum}, there is an absolute constant $C$ such that
\begin{equation}\label{eq:lp:strong-radius}
 \|F\|_6
 \le C\Bigl[\frac{\log(3+n)}{\log\log(3+n)}\Bigr]^{1/6}\|F\|_2.
\end{equation}
\end{corollary}
\begin{proof}
The representation bound \cite{ETT} gives
\begin{equation}\label{eq:lp:cardinality}
 \log(2+M_n)\le C\frac{\log(3+n)}{\log\log(3+n)}.
\end{equation}
Since $|\{z\in P_n:a_z\ne0\}|\le M_n$,
\eqref{eq:lp:strong-support} and \eqref{eq:lp:cardinality}
give \eqref{eq:lp:strong-radius}.
\end{proof}

\begin{corollary}\label{cor:flat}
Let $F$ be given by \eqref{eq:app:weighted-sum} with $M_n>0$. Then
\begin{equation}\label{eq:flat}
 \|F\|_6\le C^{1/6}\sqrt{M_n}\,\max_{z\in P_n}|a_z|.
\end{equation}
In particular,
\begin{equation}\label{eq:app:flatness}
 \max_{z\in P_n}|a_z|\le KM_n^{-1/2}\|F\|_2
 \quad\Longrightarrow\quad
 \|F\|_6\le C^{1/6}K\|F\|_2\qquad(K\ge1).
\end{equation}
\end{corollary}
\begin{proof}
Fourier orthogonality, \eqref{eq:app:weighted-sum}, and \eqref{eq:main} give
\begin{equation}\label{eq:flat:expansion}
 \|F\|_6^6
 =\sum_{\substack{z_1,\ldots,z_6\in P_n\\
                  z_1+z_2+z_3=z_4+z_5+z_6}}
       a_{z_1}a_{z_2}a_{z_3}\overline{a_{z_4}a_{z_5}a_{z_6}}
  \le J_3(P_n)\max_{z\in P_n}|a_z|^6
 \le CM_n^3\max_{z\in P_n}|a_z|^6.
\end{equation}
Taking sixth roots proves \eqref{eq:flat}, hence \eqref{eq:app:flatness}.
\end{proof}

The uniform $L^4$ bounds of Cooke and Zygmund \cite{Cooke,Zygmund}
are the classical case of Bourgain's conjecture \cite{BourgainMoment} on uniform bounds for every
finite $p$.
Theorem~\ref{thm:lp:weak} extends the classical uniform range to $p<6$.
At the endpoint $p=6$, Corollary~\ref{cor:flat} gives uniform bounds under \eqref{eq:app:flatness}.

\subsection{Triple sumsets}
\begin{corollary}\label{cor:sumsets}
Let $n_1,n_2,n_3$ be positive integers and $0<\delta_j\le1$ for $j=1,2,3$.
Suppose that
$S_j\subset P_{n_j}$ and
$|S_j|\ge\delta_jM_{n_j}>0$ for $j=1,2,3$. Then
\begin{equation}\label{eq:sumsets:density}
 C^{-1}\delta_1\delta_2\delta_3|S_1||S_2||S_3|
 \le |S_1+S_2+S_3|\le |S_1||S_2||S_3|.
\end{equation}
In particular, taking $S_j=P_{n_j}$ and $\delta_j=1$ gives
\begin{equation}\label{eq:sumsets}
 C^{-1}\prod_{j=1}^3M_{n_j}
 \le |P_{n_1}+P_{n_2}+P_{n_3}|\le\prod_{j=1}^3M_{n_j}.
\end{equation}
\end{corollary}
\begin{proof}
For $v\in\Z^2$, let $r(v)$ count the ordered representations
$v=z_1+z_2+z_3$ with $z_j\in S_j$.
Cauchy--Schwarz, Parseval and H\"older, followed by
\eqref{eq:energy:def}, \eqref{eq:energy} and \eqref{eq:main}, give
\begin{equation}\label{eq:app:mixed-energy}
 \begin{aligned}
 \Bigl(\prod_{j=1}^3|S_j|\Bigr)^2
 &=\Bigl(\sum_{v\in S_1+S_2+S_3}r(v)\Bigr)^2
 \le |S_1+S_2+S_3|\sum_vr(v)^2\\
 &=|S_1+S_2+S_3|\int_{\T^2}|f_{S_1}f_{S_2}f_{S_3}|^2
 \le |S_1+S_2+S_3|\prod_{j=1}^3J_3(S_j)^{1/3}\\
 &\le C|S_1+S_2+S_3|\prod_{j=1}^3M_{n_j}
 \le C|S_1+S_2+S_3|\prod_{j=1}^3(\delta_j^{-1}|S_j|).
 \end{aligned}
\end{equation}
Division gives the lower bound in \eqref{eq:sumsets:density}. Counting
the ordered triples gives the upper bound.
\end{proof}

By \eqref{eq:app:mixed-energy}, Bourgain's bound
\cite[Theorem~2.2]{KKW} gives
$|P_n+P_n+P_n|/M_n^2\to\infty$ as $M_n\to\infty$, while
Bombieri--Bourgain \cite{BB} gives a lower bound of
order $M_n^{5/2}$.
The estimates of Cushman--Demeter--Wu \cite[Corollary~1.3]{CDW}
and Li--Yung \cite[Theorem~1.1]{LY} improve this to lower bounds
of orders $M_n^{3-\varepsilon}$ and $M_n^3/\log(2+M_n)$,
respectively.
Corollary~\ref{cor:sumsets} removes the remaining logarithm and
establishes the optimal cubic order for lattice circles.
It also gives the corresponding product bound for three possibly
different radii and subsets of fixed positive relative density,
with the explicit dependence on the densities in
\eqref{eq:sumsets:density}.

\subsection{Mean-square convergence of observables}
\begin{corollary}\label{cor:observables}
For each $n$ with $P_n\ne\varnothing$, let
$(\phi_{n,j})_{j=1}^{M_n}$ be any orthonormal basis of the eigenspace
of $-\Delta$ on $\T^2$ with eigenvalue $4\pi^2n$.
Every fixed $V\in L^{6/5}(\T^2)$ satisfies
\begin{equation}\label{eq:observables:convergence}
 \frac1{M_n}\sum_{j=1}^{M_n}
 \Bigl|\int_{\T^2}V|\phi_{n,j}|^2-\int_{\T^2}V\Bigr|^2
 \longrightarrow0
 \qquad(M_n\to\infty),
\end{equation}
uniformly over the choice of these bases.
\end{corollary}
\begin{proof}
Use the Fourier coefficients in \eqref{eq:fourier:coefficients}.
The sum of the squared diagonal entries of a matrix is at most its
squared Hilbert--Schmidt norm. Applying this to the form
$(u,w)\mapsto\int_{\T^2}Vu\overline w-\widehat V(0)\int_{\T^2}u\overline w$
on the eigenspace, and using unitary invariance, gives
\begin{equation}\label{eq:observables:variance}
 \frac1{M_n}\sum_{j=1}^{M_n}
 \Bigl|\int_{\T^2}V|\phi_{n,j}|^2-\widehat V(0)\Bigr|^2
 \le S_n(V):=\frac1{M_n}
       \sum_{\substack{z,w\in P_n\\z\ne w}}|\widehat V(z-w)|^2.
\end{equation}
By \eqref{eq:fourier:sum}, \eqref{eq:energy}, and \eqref{eq:main},
\begin{equation}\label{eq:observables:kernel}
 K_n(x):=\frac{|f_{P_n}(x)|^2}{M_n},
 \qquad
 \|K_n\|_3=\frac{J_3(P_n)^{1/3}}{M_n}\le C^{1/3}.
\end{equation}
Fourier expansion, H\"older's inequality, and Young's convolution
inequality give
\begin{equation}\label{eq:observables:young}
 \begin{split}
 S_n(V)+|\widehat V(0)|^2
 &=\int_{\T^2}\int_{\T^2}
       V(x)\overline{V(y)}K_n(x-y)\,dx\,dy\\
 &\le\|K_n\|_3\|V\|_{6/5}^2
 \le C^{1/3}\|V\|_{6/5}^2.
 \end{split}
\end{equation}
For $v\ne0$, the conditions $z,z-v\in P_n$ imply
$2z\cdot v=|v|^2$, so there are at most two such $z$.
Consequently Parseval's identity yields
\begin{equation}\label{eq:observables:L2}
 S_n(W)\le\frac2{M_n}\sum_{v\ne0}|\widehat W(v)|^2
          \le\frac2{M_n}\|W\|_2^2
 \qquad(W\in L^2(\T^2)).
\end{equation}
For any trigonometric polynomial $W$, the triangle inequality in
\eqref{eq:observables:variance}, followed by
\eqref{eq:observables:young} and \eqref{eq:observables:L2}, gives
\[
 S_n(V)^{1/2}
 \le C^{1/6}\|V-W\|_{6/5}
       +(2/M_n)^{1/2}\|W\|_2.
\]
First let $M_n\to\infty$, then approximate $V$ in $L^{6/5}$ by $W$.
Equation~\eqref{eq:observables:variance} proves
\eqref{eq:observables:convergence}, uniformly over the bases.
\end{proof}

Hezari--Rivi\`ere \cite[Proposition~3.2]{HR} proved an
$O(M_n^{-1})$ mean-square bound for fixed $L^2$ observables on individual
eigenspaces, uniformly over orthonormal bases. The energy estimates in \cite{CDW,LY} already imply
convergence for fixed $V\in L^p(\T^2)$ with $p>6/5$.
Corollary~\ref{cor:observables} reaches the endpoint $p=6/5$ along every sequence with
$M_n\to\infty$, uniformly over orthonormal bases.

\subsection{Nodal length of arithmetic random waves}
Fix $n\ge1$ with $P_n\ne\varnothing$, and let
\begin{equation}\label{eq:app:arw-data}
 \lambda_n=4\pi^2n,\qquad
 P_n=\bigsqcup_{j=1}^{M_n/2}\{z_j,-z_j\}.
\end{equation}
On a probability space $\Omega$, choose coefficients so that the $M_n$
real variables
\begin{equation}\label{eq:app:gaussian-coefficients}
 \sqrt2\Re a_{z_j},\ \sqrt2\Im a_{z_j}\sim N(0,1),\qquad
 a_{-z_j}=\overline{a_{z_j}}\quad(1\le j\le M_n/2)
\end{equation}
are independent standard Gaussians. Let $\mathbb E$ and
$\operatorname{Var}$ denote expectation and variance. The arithmetic
random wave
\begin{equation}\label{eq:arw}
 T_n(x)=\frac1{\sqrt{M_n}}\sum_{z\in P_n}a_z e^{2\pi i z\cdot x}
\end{equation}
is real, centered, and of variance one. Let
\begin{equation}\label{eq:arw:parameters}
 r_n(x)=\mathbb E[T_n(x)T_n(0)]=M_n^{-1}f_{P_n}(x),\qquad
 \eta_n=\frac1{M_n}\sum_{z\in P_n}\cos(4\arg z)\in[-1,1].
\end{equation}
Let $\mathcal L_n$ be the nodal length of $T_n$ on $\T^2$, and let
$\mathcal L_n[q]$ be its orthogonal projection in $L^2(\Omega)$ onto the
$q$th Wiener chaos, spanned by products of Hermite polynomials of total
degree $q$ in the independent variables in \eqref{eq:app:gaussian-coefficients}.
See \cite[\S2.1, Definition~2.1]{MPRW}.

\begin{corollary}\label{cor:nodal}
Uniformly over all $n\ge1$ with $P_n\ne\varnothing$,
\begin{equation}\label{eq:nodal:variance}
 \operatorname{Var}(\mathcal L_n)
 =\frac{1+\eta_n^2}{512}\frac{\lambda_n}{M_n^2}
     +O\!\Bigl(\frac{\lambda_n}{M_n^3}\Bigr),
\end{equation}
and
\begin{equation}\label{eq:nodal:tail}
 \mathbb E\left|\mathcal L_n-\mathbb E\mathcal L_n
                         -\mathcal L_n[4]\right|^2
 \le C\frac{\lambda_n}{M_n^3}.
\end{equation}
Consequently, as $M_n\to\infty$,
\begin{equation}\label{eq:nodal:normalized}
 \left\|
 \frac{\mathcal L_n-\mathbb E\mathcal L_n}
                {\sqrt{\operatorname{Var}(\mathcal L_n)}}
 -\frac{\mathcal L_n[4]}
                {\sqrt{\operatorname{Var}(\mathcal L_n[4])}}
 \right\|_{L^2(\Omega)}\le C M_n^{-1/2}.
\end{equation}

\end{corollary}

\begin{proof}
The chaos expansion \cite[Proposition~3.2]{MPRW} converges in
$L^2(\Omega)$ and has no odd or second-order terms. Thus
\begin{equation}\label{eq:app:chaos-expansion}
 \mathcal L_n-\mathbb E\mathcal L_n
 =\mathcal L_n[4]+\mathcal L_n[\ge6],\qquad
 \mathcal L_n[\ge6]:=\sum_{q\ge3}\mathcal L_n[2q].
\end{equation}
By orthogonality and \cite[Lemma~2, (27)]{PR}, with $K=3$,
\begin{equation}\label{eq:app:tail-input}
 V_{\ge6}:=\mathbb E|\mathcal L_n[\ge6]|^2
 =\sum_{q\ge3}\operatorname{Var}(\mathcal L_n[2q])
 \le C\lambda_n\int_{\T^2}r_n(x)^6\,dx.
\end{equation}
Since $P_n=-P_n$, equations \eqref{eq:arw:parameters},
\eqref{eq:energy}, and \eqref{eq:main} give
\begin{equation}\label{eq:app:correlation-sixth}
 \int_{\T^2}r_n(x)^6\,dx
 =\frac{J_3(P_n)}{M_n^6}\le C M_n^{-3}.
\end{equation}
Together with \eqref{eq:app:chaos-expansion}--\eqref{eq:app:tail-input},
this proves \eqref{eq:nodal:tail}.

Let $V_4=\operatorname{Var}(\mathcal L_n[4])$.
The fourth-chaos representation \cite[Lemma~4, (47)]{PR} and the Gaussian
moment formula give
\begin{equation}\label{eq:nodal:fourth}
 V_4=\frac{\lambda_n}{512M_n^2}
       \bigl(1+\eta_n^2+O(M_n^{-1})\bigr),
\end{equation}
with an absolute error constant. Orthogonality gives
\begin{equation}\label{eq:app:variance-split}
 \operatorname{Var}(\mathcal L_n)=V_4+V_{\ge6},\qquad
 \mathbb E[\mathcal L_n[4]\mathcal L_n[\ge6]]=0.
\end{equation}
Equations \eqref{eq:app:tail-input}--\eqref{eq:app:variance-split} yield
\eqref{eq:nodal:variance}. Since $1+\eta_n^2\ge1$, for large $M_n$,
\begin{equation}\label{eq:app:variance-ratio}
 V_4\ge\frac{\lambda_n}{1024M_n^2},\qquad
 \frac{V_{\ge6}}{V_4}\le C M_n^{-1}.
\end{equation}
Finally, \eqref{eq:app:chaos-expansion} and \eqref{eq:app:variance-split} give
\begin{equation}\label{eq:app:normalized-chaos-distance}
 \left\|
 \frac{\mathcal L_n-\mathbb E\mathcal L_n}{\sqrt{V_4+V_{\ge6}}}
 -\frac{\mathcal L_n[4]}{\sqrt{V_4}}
 \right\|_{L^2(\Omega)}^2
 =2\Bigl(1-\sqrt{\frac{V_4}{V_4+V_{\ge6}}}\Bigr)
  \le\frac{V_{\ge6}}{V_4}\le C M_n^{-1}.
\end{equation}
Taking square roots proves \eqref{eq:nodal:normalized}.
\end{proof}

Rudnick--Wigman \cite[Theorem~6.1]{RW} bounded the variance, and
Krishnapur, Kurlberg and Wigman \cite[Theorem~1.1]{KKW} determined its
leading term. Marinucci, Peccati, Rossi and Wigman
\cite[Proposition~2.3 and Theorem~1.1]{MPRW} proved fourth-chaos dominance
and non-Gaussian limits when $M_n\to\infty$ and $|\eta_n|$ converges,
with quantitative refinements in \cite[Proposition~2 and Theorem~2]{PR}.
Via \eqref{eq:app:tail-input} and \eqref{eq:nodal:fourth}, the energy bounds
in \cite{BB,CDW,LY} give variance remainders of orders
$\lambda_nM_n^{-5/2}$, $\lambda_nM_n^{-3+\varepsilon}$, and
$\lambda_nM_n^{-3}\log(2+M_n)$, respectively.
Corollary~\ref{cor:nodal} removes the logarithm and yields the
$O(M_n^{-1/2})$ approximation \eqref{eq:nodal:normalized} uniformly as
$M_n\to\infty$, without assuming convergence of $\eta_n$.

\end{document}